\documentclass[12pt,a4paper]{article}
\usepackage{amsthm,amsmath,amssymb}
\usepackage{mathrsfs}
\usepackage{indentfirst}
\usepackage{tensor}
\usepackage{bbm}
\usepackage{xcolor}
\usepackage{authblk}
\allowdisplaybreaks[4]
\numberwithin{equation}{section}
\newtheorem{thm}{Theorem}[section]
\newtheorem{defn}[thm]{Definition}

\newtheorem{lem}[thm]{Lemma}
\newtheorem{prop}[thm]{Proposition}
\newtheorem{rmk}[thm]{Remark}
\usepackage{hyperref}

\begin{document}
\title{\Large \bf Ergodicity of Stochastic two-phase Stefan problem driven by pure jump L\'{e}vy noise
\footnotetext{E-mail addresses: gxt@mail.ustc.edu.cn (Xiaotian Ge); sjshang@ustc.edu.cn (Shijie Shang); 

\qquad\qquad\qquad\qquad zhaijl@ustc.edu.cn (Jianliang Zhai); tusheng.zhang@manchester.ac.uk(Tusheng Zhang).}}

\author[a]{Xiaotian Ge}
\author[a]{Shijie Shang}
\author[a]{Jianliang Zhai}
\author[a,b]{Tusheng Zhang}
\affil[a]{\small School of Mathematical Sciences, University of Science and Technology of China\\

Hefei, Anhui 230026, China.}
\affil[b]{\small Department of Mathematics, University of Manchester, Oxford Road\\

Manchester, M13 9PL, UK.}
%
\maketitle

\begin{center}
\begin{minipage}{130mm}
{\bf Abstract.} In this paper, we consider stochastic two-phase Stefan problem driven by general jump L\'{e}vy noise. We first obtain the existence and uniqueness of the strong solution and then establish the ergodicity of the stochastic Stefan problem. Moreover, we give a precise characterization of the support of the invariant measures which provides  the regularities of the stationary solutions of the stochastic free boundary problems.

\vspace{3mm} {\bf Keywords.} Stochastic two-phase Stefan problem; Pure jump L\'{e}vy noise; Ergodicity; Strong solution; Yosida approximation; Irreducibility;
$e$-property.
\end{minipage}
\end{center}

\section{Introduction}
In this paper, we are concerned with the ergodicity of stochastic two-phase  Stefan problem driven by pure jump L$\rm\acute{e}$vy noise:

\begin{equation}\label{1.0} \begin{cases}
d\theta(t,\xi)-\Delta\theta(t,\xi)dt=dL_t\ {\rm in}\ \{(t,\xi)\in [0,T]\times D: \theta(t,\xi)>0\},\\
d\theta(t,\xi)-a\Delta\theta(t,\xi)dt=dL_t\ {\rm in}\ \{(t,\xi)\in [0,T]\times D: \theta(t,\xi)<0\},\\
\frac{\partial \theta_+}{\partial \nu}(t,\xi)-a\frac{\partial \theta_-}{\partial \nu}(t,\xi)=-\rho \ 
{\rm on}\ \{(t,\xi)\in [0,T]\times D: \theta(t,\xi)=0\},\\
\theta_+(t,\xi)=\theta_-(t,\xi)=0\ {\rm on}\ \{(t,\xi)\in [0,T]\times D: \theta(t,\xi)=0\},\\ 
\theta(0,\xi)=\theta_0(\xi)\ {\rm in}\ D, \theta(t,\xi)=0\ {\rm on}\ [0,T]\times\partial D.
\end{cases}\end{equation}
System (\ref{1.0}) models the melting or freezing process of ice-water mixture in domain $D\subset\mathbb{R}^n$, with a random heating source $L=(L_t)_{t\geq0}$, which is a pure jump L$\rm\acute{e}$vy noise, and we use $\theta(t,\xi)$ to represent the tempurature of ice or water. Here $a$ means the thermal conductivity of ice, and $\rho$ means the latent heat, both $a$ and $\rho$ are positive constants. Denote $\Gamma_t=\{\xi\in D:\theta(t,\xi)=0\}$ as the free boundary, $\theta_+,\theta_-$ are limits of $\theta$ with respect to water region $D_t^+:=\{\xi\in D:\theta(t,\xi)>0\}$ and ice region $D_t^-:=\{\xi\in D:\theta(t,\xi)>0\}$ on $\Gamma_t$, and $\partial\theta_+/\partial\nu,\partial\theta_-/\partial\nu$ are outwards normal derivatives to the free boundary $\Gamma_t$ with respect to region $D_t^+$,$D_t^-$. Clearly $D=D_t^+\cup D_t^-\cup \Gamma_t$ for all $t\in[0,T]$.

\vskip 0.3cm
The Stefan problem  is among the most classical and well-known free boundary problems. The two-phase Stefan problem  models the evolution of temperature in a material with two thermodynamical states, say solid and liquid. The solid-liquid phase transition occurs at a given constant temperature, which we set at zero. Through the so-called ``enthalpy formulation'' of the Stefan problem, we are lead to the following  weak form of the classical Stefan problem:

\begin{equation}\label{1.1} \begin{cases} dX(t)=\Delta \beta(X(t))dt+dL_t\ {\rm in}\ [0,T]\times D,\\
X(0)=x\ {\rm in}\ D,\,\beta(X(t)
)=0\ {\rm on} \ [0,T]\times \partial D.\end{cases}\end{equation}
Here $\beta:\mathbb{R}\to\mathbb{R}$ is the inverse function of the so-called ``enthalpy function'', which has the expression:
\begin{equation}\label{1.2}
\beta(r)=\begin{cases}ar,\, ~~~~~~r<0,\\ 0,\ ~~~~~~~r\in[0,\rho],\\ r-\rho, ~~~r>\rho.\end{cases}
\end{equation}

 We refer the readers to \cite{MR1423808} for the history of Stefan problem and the explanation on the relation between the classical Stefan problem and its weak form (\ref{1.1}). See  also \cite{MR2582280} and \cite{MR1154310} for discussions on the deterministic Stefan problem with heat source, and \cite{MR4476224}, \cite{MR4440074}, \cite{MR4695505} for some recent progresses on this topic.

In the past two decades, there are a few works on the well-posedness of Stefan problem perturbed by stochastic heat source. In the pioneering work \cite{MR1942322}, the authors considered stochastic two-phase Stefan problem driven by Wiener process, and obtained the existence and uniqueness of the strong solution with initial data $x\in L^2(D)$. Afterwards, several papers treated the stochastic two-phase  Stefan problem as a special case of stochastic porous media equations, and these works focused on the well-posedness of generalized solutions whose initial data takes values in the space  $H^{-1}(D)$; see e.g., \cite{MR3410409}, \cite{MR3560817}.
Here the generalized solution is obtained as the limit of strong solutions, see Section 4.1 for details. In \cite{MR4098221}, a large class of doubly nonlinear stochastic evolution equations was considered, which includes the stochastic two-phase Stefan problem with multiplicative noise. Recently, the numerical analysis for the stochastic two-phase Stefan problem driven by multiplicative noise was discussed in \cite{droniou2023numerical}. There are also some other types of stochastic Stefan problem, see, for example, \cite{MR2993014} for stochastic one-phase  Stefan problem, \cite{MR3554430}, \cite{MR3843831} for one dimensional Stefan-type stochastic moving boundary problem, \cite{MR4399135} for stochastic one-phase Stefan problem with Gibbs-Thomson condition, and \cite{MR4046526} for stochastic two-phase Stefan problems with reflection. The driving noise considered in the above mentioned works are Gaussian noise. 

In this paper, we focus on the well-posedness and the ergodicity of the stochastic Stefan problem (\ref{1.1}) driven by general L\'{e}vy noise. In the setting of Gaussian noise, \cite{MR1942322, MR3560817} considered the existence of invariant measures of the  transition semigroup generated by generalized solutions of the following system:
\begin{equation}\label{1.3} \begin{cases} dX(t)=\Delta \beta(X(t))dt+\sqrt{Q}dW_t\ {\rm in}\ [0,T]\times D,\\
X(0)=x\ {\rm in}\ D,\,\beta(X(t))=0\ {\rm on} \ [0,T]\times \partial D,\end{cases}\end{equation}
where $\sqrt{Q}$ is a Hilbert-Schmidt operator, and $(W_t)_{t\geq0}$ is a cylindrical Wiener process. 
The methods used in \cite{MR1942322} and \cite{MR3560817} are different. 
In \cite{MR3560817}, the authors applied the classical Krylov-Bogoliubov theorem to prove the existence of invariant measures. And in \cite{MR1942322} the authors used the Yosida approximations, which also gave a characterization of the support of the invariant measures. However,
the uniqueness of invariant measure for stochastic two-phase Stefan problem (\ref{1.3}) was left as an open problem as mentioned in \cite{MR3560817}.
For the long time asymptotics of the Stefan problem,
 we also mention the reference  \cite{MR3296833} where the author proved that the solution of the stochastic two-phase Stefan problem driven by linear multiplicative Gaussian noise converges to $0$ in probability as $t\to\infty$.

In this paper, we will establish the ergodicity of equation (\ref{1.1}). There are two main differences between the results mentioned above for the Gaussian case and the results in this paper. First, we study the transition semigroup generated by the strong solutions.
 Second, the driving noise $(L_t)_{t\geq0}$ in  (\ref{1.1}) can be quite general pure jump L$\rm\acute{e}$vy process including  $\alpha$-stable processes with $\alpha\in(0,2)$ and compound Poisson processes, which is somehow surprising.  An important
novel result of this article is the precise characterization of the support of the invariant measure, which guarantees the regularity of the stationary solutions. The study of the regularity is important, because it could be used to see whether the phase transition occurs continuously across the interface, which is the key topic of central concern in the study of free boundary problems, see \cite{MR0683353}, \cite{MR4476224}, \cite{MR4695505}.

\vskip 0.3cm

We now describe the main results in this paper and ideas of their proofs.
\vskip 0.3cm

$\bullet$ Existence and uniqueness of the solution. Using a family of Yosida-type approximating equations (see (\ref{3.7})), we  prove that equation (\ref{1.1}) admits a unique strong solution for any initial data $x\in L^2(D)$, denoted by $(X(t,x))_{t\geq0}$, living in the space $L^\infty_{loc}([0,\infty);L^2(D))$. And the family of the solutions $(X(t,x))_{t\geq0}, x\in L^2(D)$  generates a Markov transition semigroup $(Q_t)_{t\geq0}$ on the space $L^2(D)$. We are not able to show that the strong solution  $(X(t,x))_{t\geq0}, x\in L^2(D)$ is  $\rm c\grave{a}dl\grave{a}g$ in $L^2(D)$. However, the generalized solution $(X(t,x))_{t\geq0}, x\in H^{-1}(D)$ of (\ref{1.1}) obtained as limits of strong solutions lives in the Skorohod   space $D([0,\infty);H^{-1}(D))$ (see Section 4.1) , and they generate a Feller Markov transition semigroup $(P_t)_{t\geq0}$ on the space $H^{-1}(D)$.

\vskip 0.3cm

$\bullet$ Uniqueness of invariant measures. Since every  invariant measure of $(Q_t)_{t\geq0}$ could be extended to be an
 invariant measure of $(P_t)_{t\geq0}$, to prove the uniqueness of invariant measure of $(Q_t)_{t\geq0}$, it is sufficient to prove that $(P_t)_{t\geq0}$ has at most one invariant measure. In \cite{2207.11488}, some of the co-authors of this paper provided an effective criterion  on the irreducibility of stochastic partial differential equations driven by  pure jump L$\rm\acute{e}$vy noise. Using this criterion, we obtain the irreducibility of the generalized solution $(X(t,x))_{t\geq0}, x\in H^{-1}(D)$ of (\ref{1.1}). Combining this with the so-called $e$-property of the generalized solutions, we show that  $(P_t)_{t\geq0}$ admits at most one invariant measure, see Proposition \ref{[4.8]}.
\vskip 0.3cm
$\bullet$ Existence and the support of the invariant measures. For the existence, we will employ a Yosida-type approximations inspired by  \cite{MR1942322}. We construct an invariant measure of $(Q_t)_{t\geq0}$ by taking weak convergence limit of the invariant measures of approximating equations, see Proposition \ref{[4.13]}.
We stress that one can also prove the existence of invariant measures of $(Q_t)_{t\geq0}$ by showing that $(P_t)_{t\geq0}$ admits an invariant measure supported on $L^2(D)$ using the Krylov-Bogoliubov criteria. However, in this way we are not able to derive a more precise support of the invariant measure of $(Q_t)_{t\geq0}$. We will show that the support of the invariant measure is on the set $D(A):=\left\{x\in L^2(D):\, \beta(x)\in H_0^1(D)\right\}$. To this end,  we will derive a number of a prior estimates for the invariant measures of the approximating equations and provide several properties of the  Yosida-type approximation operators. The main difficulty we need to deal with is that the driving process $(L_t)_{t\geq0}$ may not be square integrable. Therefore, the details will be quite different from the case of Wiener noise, which strongly rely on the square integrability of the solutions and some peculiar properties of Wiener process.

Finally, we point out that there are not many papers studying  ergodicity of SPDEs driven by pure jump L\'evy noise, and we refer to \cite{MR2727320}, \cite{MR2773026}, \cite{MR3164597},\cite{MR3490494}, \cite{MR3554893}, \cite{MR3606763}, \cite{MR3780697}, \cite{MR4067297} and the literatures therein.

\vskip 0.3cm
The organization of this paper is as follows. In Section 2 we introduce the Stefan problem (\ref{1.1}), give the definition of the strong solution and state the main results, including the well-posedness and ergodicity. Moreover, several specific examples of pure jump L$\rm\acute{e}$vy noise will also be given. In Section 3, we will prove the existence, uniqueness and the Markov property of the strong solution.
In Section 4 we define the generalized solutions of (\ref{2.3}) and  prove the uniqueness, existence of the invariant measures. In addition, we give a precise characterization of the support of the invariant measures.

Throughout this paper, the symbol $C$ denotes a generic positive constant whose value may change from line to line.

\section{Framework and the main results}
In this section, we will introduce the stochastic Stefan problem and state the main results, including the well-posedness of strong solutions and the ergodicity of the solutions.
\subsection{Stochastic Stefan problem}

\indent Let $(\Omega,\mathcal{F},\mathbb{P})$ be a complete probability space with filtration $\mathbb{F}=(\mathcal{F}_t)_{t>0}$ satisfying the usual conditions. Let $D$ be a bounded domain in $\mathbb{R}^n$ with smooth boundary $\partial D$. We denote by $L^2(D)$ the space of all square integrable functions on $D$ and $H^1_0(D)$  the space of all functions belonging to Sobolev space $W^{1,2}(D)$ with zero trace. The $L^2$-norm $|\cdot|_2$ and $H^1_0(D)$-norm $||\cdot||_1$ are defined respectively as follows:
$$|u|_2:=\left(\int_D |u(\xi)|^2d\xi\right)^{\frac{1}{2}},\,||u||_1:=\left(\int_D |\nabla u(\xi)|^2d\xi\right)^{\frac{1}{2}}.$$
Let $H^{-1}(D)$ be the dual space of $H_0^1(D)$, and $||\cdot||_{-1}$ be the $H^{-1}(D)$-norm. For $H=H^{-1}(D),L^2(D)$ or $H_0^1(D)$, we denote by $\left<\cdot,\cdot\right>_H$ the inner product on $H$. By the dualization between $H_0^1(D)$ and $H^{-1}(D)$, we have
$$\tensor*[_{H^{-1}}]{\left<x,y\right>}{_{H_0^1}}=\left<x,y\right>_{L^2},\; \forall x\in L^2(D),y\in H_0^1(D).$$

Let $\mathcal{B}(H)$ be the Borel $\sigma$-field on $H$, and we use $B_b(H)$,$C_b(H)$,${\rm Lip}_b(H)$ to denote the space of all bounded $\mathcal{B}(H)$-measurable functions, bounded continuous functions and bounded Lipschitz continuous functions respectively. We have ${\rm Lip}_b(H)\subset C_b(H)\subset B_b(H)$ with densely embedding. Since the embedding $L^2(D)\subset H^{-1}(D)$ is continuous, for any $f\in C_b(H^{-1}(D))$, $f|_{L^2}$ (the restriction of $f$ on $L^2(D)$) belongs to $C_b(L^2(D))$, and we regard this fact as $C_b(H^{-1}(D))\subset C_b(L^2(D))$.

\indent We denote by $\mathcal{M}_1(H)$ the space of all Borel probability measures on $(H, \mathcal{B}(H))$. Since any $\pi\in\mathcal{M}_1(L^2(D))$ can be extended to be an element $\tilde{\pi}\in\mathcal{M}_1(H^{-1}(D))$ by letting $\tilde{\pi}(O)=0$ for any $O\in\mathcal{B}(H^{-1}(D))$ with $O\subset H^{-1}(D)\backslash L^2(D)$, we write this fact as $\mathcal{M}_1(L^2(D))\subset\mathcal{M}_1(H^{-1}(D))$.     \\
\indent Let $\Delta$ be the Laplacian operator. It is well-known that $-\Delta$ can be extended to a bounded linear operator from $H_0^1(D)$ to $H^{-1}(D)$ and
$$\tensor*[_{H^{-1}}]{\left<-\Delta x,y\right>}{_{H_0^1}}=\left<\nabla x,\nabla y\right>_{L^2}
=\left<x,y\right>_{H_0^1},\; \forall x,y\in H^1_0(D).$$
Moreover, $-\Delta$ is an isomorphism from $H_0^1(D)$ onto $H^{-1}(D)$. For $\forall x,y\in H^{-1}(D)$ we have
$$\left<x,y\right>_{H^{-1}}=\left<(-\Delta)^{-\frac{1}{2}}x,(-\Delta)^{-\frac{1}{2}}y\right>_{L^2(D)}
=\left<(-\Delta)^{-1}x,(-\Delta)^{-1}y\right>_{H_0^1(D)},$$
and when $x\in H_0^1(D),y\in L^2(D)$ we have $\left<-\Delta x,y\right>_{H^{-1}}=\left<x,y\right>_{L^2}$.
\vskip 0.2cm

\indent Next, we introduce the pure jump L$\rm\acute{e}$vy noise. Let $Z=L^2(D)$, $\nu$ be a $\sigma$-finite measure on $(Z,\mathcal{B}(Z))$ with $\int_{Z}(|z|_2^2\wedge 1)\nu(dz)<\infty$, where ``$\sigma$-finite" means there exists a sequence of subsets $Z_n\in\mathcal{B}(Z)$ such that $Z_n\uparrow Z$ and $\nu(Z_n)<\infty$. Let $N:\mathcal{B}(Z\times\mathbb{R}_+)\times\Omega\to\bar{\mathbb{N}}$ be the time homogeneous Poisson random measure on $(Z,\mathcal{B}(Z))$ with intensity measure $\nu$, where $\bar{\mathbb{N}}=\mathbb{N}\cup\{0,+\infty\}$. Let $\tilde{N}(dzds):=N(dzds)-\nu(dz)ds$ be the compensated Poisson random measure associated to $N$. The pure jump L$\rm\acute{e}$vy process $L_t$ on $Z$ can be decomposed as
\begin{align}\label{240416.2003}
    L_t=\int_0^t\int_{|z|_2\leq1}z\tilde{N}(dzds)+\int_0^t\int_{|z|_2>1}zN(dzds) =:\tilde{L}_t+\hat{L}_t.
\end{align}

\indent Now we go back to the stochastic Stefan problem stated in (\ref{1.1}), that is,
\begin{equation}\label{2.2} \begin{cases} dX(t)=\Delta \beta(X(t))dt+dL_t\ {\rm in}\ [0,T]\times D,\\
X(0)=x\ {\rm in}\ D,\,\beta(X)=0\ {\rm on} \ [0,T]\times \partial D,\end{cases}\end{equation}
where $\beta$ was defined in (\ref{1.2}), and $x$ is an element in $L^2(D)$.

Inspired by \cite{MR2582280} and \cite{MR1942322}, we will reformulate system (\ref{1.1}) into a more simplified weak form which can be dealt with. Remark that $\beta$ can be seen as an operator from $L^2(D)$ to $L^2(D)$ or from $H^1_0(D)$ to $H^1_0(D)$, since $\beta(r)$ is global Lipschitz continuous and has bounded weak derivative. Hence we define the operator $Ax:=-\Delta\beta(x)\in H^{-1}(D)$ with domain
$$D(A):=\left\{x\in L^2(D):\, \beta(x)\in H_0^1(D)\right\}.$$
Clearly an element $x$ in $L^2(D)$ belongs to $D(A)$ is equivalent to $||Ax||_{-1}<\infty$, and $H_0^1(D)\subset D(A)$, hence $D(A)$ is dense in $H^{-1}(D)$. Now (\ref{2.2}) can be rewritten as follow:
\begin{equation} \label{2.3}\begin{cases}dX(t)=-AX(t)dt+dL_t,t\in[0,T],\\ X(0)=x\in L^2(D).
\end{cases}\end{equation}
\indent In this paper we will focus on equation (\ref{2.3}), and we will prove that (\ref{2.3}) has a unique invariant measure.
For this purpose, we need the following two conditions:\\
(\textbf{C1}): for some $\alpha\in(0,2)$, we have $\int_{|z|_2>1} |z|_2^\alpha \nu(dz)<+\infty$;\\
(\textbf{C2}): the set $$H_0:=\left\{\sum_{i=1}^{n}m_ia_i:\;n,m_1,...,m_n\in\mathbb{N},\;a_1,...,a_n\in S_\nu\right\}$$
is dense in $L^2(D)$, where
$$S_{\nu}:=\left\{x\in L^2(D):\; \nu(G)>0\; {\rm for\; any\; open\; set}\; G\subset L^2(D)\; {\rm containing}\; {x}\right\}.$$
Condition (\textbf{C1}) on ``big jumps" of $L_t$ will be used to prove the existence of invariant measures. Condition (\textbf{C2}) will be used to obtain the irreducibility of the Stefan problem driven by pure jump noise, which is further used to show the uniqueness of invariant measures.

\subsection{Main results}

In this subsection, we will state the main results of this paper. Let us begin with the definition of solutions to equation (\ref{2.3}), and give the main results on existence and uniqueness of solutions to equation (\ref{2.3}).

\begin{defn}\label{[2.2]}  We call the process $X=(X(t))_{t\in[0,T]}$ a strong solution of equation (\ref{2.3}) if \\
$\rm (i)$ $X$ is an $H^{-1}(D)$-valued $\rm c\grave{a}dl\grave{a}g$  $\mathbb{F}$-adapted process;\\
$\rm (ii)$ $X$ is $L^2(D)$-valued and $X\in L^{\infty}([0,T]; L^2(D))$, $\mathbb{P}$-a.s.;\\
$\rm (iii)$ $\beta(X)\in L^2([0,T]; H_0^1(D))$, $\mathbb{P}$-a.s.;\\
$\rm (iv)$ $X(t)=x+\int_0^t\Delta\beta(X(s))ds+L_t$, $\forall\, t\in[0,T]$, $\mathbb{P}$-a.s.
\end{defn}
\begin{thm}\label{[2.3]} For $x\in L^2(D)$, equation (\ref{2.3}) admits a unique strong solution. Moreover, the strong solution $X$ forms a Markov process with state space $L^2(D)$.\end{thm}

\indent For any $x\in L^2(D)$, let $X=(X(t,x))_{t\geq0}$ be the solution to (\ref{2.3}) with initial data $x$.
For any $\varphi\in C_b(L^2(D))$ and $t>0$, let $Q_t\varphi(x):=\mathbb{E}[\varphi(X(t,x))]$, then $(Q_t)_{t\geq0}$ is the transition semigroup of $X$. Let $(Q^*_t)_{t\geq0}$ be the dual semigroup of
$(Q_t)_{t\geq0}$ on $\mathcal{M}_1(L^2(D))$, that is, for any $B\in \mathcal{B}(L^2(D))$ and any
$\mu\in\mathcal{M}_1({L^2(D)})$,
$$Q^*_t\mu(B):=\int_{L^2(D)}Q_t \mathbbm{1}_B d\mu=\int_{L^2(D)}\mathbb{P}(X(t,x)\in B)\mu(dx).$$
We call a probability measure $\mu$ on $\left(L^2(D),\mathcal{B}(L^2(D))\right)$ an invariant measure of $(Q_t)_{t\geq0}$ if  $\mu$ satisfies
$Q^*_t\mu=\mu$, $\forall\, t>0$. To be precise, we have for any $\phi\in B_b(L^2(D))$ and $t>0$,
$$\int_{L^2(D)}Q_t\phi(x)d\mu=\int_{L^2(D)}\phi(x)d\mu .$$ Next theorem is the main result on the invariant measures to (\ref{2.3}).
\begin{thm}\label{[2.4]} Under conditions (\textbf{C1}) and (\textbf{C2}), $(Q_t)_{t\geq0}$ admits a unique invariant measure $\mu$, which is supported on $D(A)$. Moreover, for any $\widetilde{\nu}\in\mathcal{M}_1(L^2(D))$ with $\int_{L^2(D)}|x|^\alpha_2 d\widetilde{\nu}<\infty$, we have as $T\to\infty$,
$$\frac{1}{T}\int_0^T Q^*_s\widetilde{\nu} ds\Rightarrow \mu$$
in weak topology of $\mathcal{M}_1(H^{-1}(D))$, here ``$\Rightarrow$" means the weak convergence of probability measures.
\end{thm}

\subsection{Examples of pure jump noise}

In this subsection we will provide some examples of pure jump L$\rm\acute{e}$vy noise which satisfy conditions (\textbf{C1}) and (\textbf{C2}) imposed in Theorem \ref{[2.4]}.
For more examples, we refer to  \cite{2207.11488}. 

\textbf{Example 1: Cylindrical L\'{e}vy process.}
For $H=L^2(D)$, let $\{e_i\}_{i\in\mathbb{N}}$ be an orthonormal basis of $H$, let $\{L_i(t)\}_{i\in\mathbb{N}}$ be a sequence of mutually independent one dimensional pure jump L\'{e}vy process with the same intensity measure $\mu$, 
then for a sequence of non zero real numbers $\{\beta_i\}_{i\in\mathbb{N}}$,
$$L(t)=\sum_{i=1}^\infty \beta_i L_i(t) e_i,\,\forall\, t\geq0$$
is called a cylindrical L\'{e}vy process. Let
$$
S_\mu:=\left\{r\in \mathbb{R}:\; \mu(G)>0\; {\rm for\; any\; open\; set}\; G\subset \mathbb{R}\; {\rm containing}\; {r}\right\}.
$$
Suppose that there exists some $\theta\in(0,2]$ such that
$$\int_{|x|>1} |x|^\theta\mu(dx)+\sum_{i=1}^\infty |\beta_i|^\theta<\infty,$$
and
there exist $a<0$ and $b>0$ with $a,b\in S_\mu$ such that $a/b$ is an irrational number,
then the intensity measure of $L(t)$ satisfies conditions (\textbf{C1}) and (\textbf{C2}) with $\alpha=\theta$. If $\mu(dx)=|x|^{-1-\alpha}dx$ with $\alpha\in(0,2)$, then $(L(t))_{t\geq0}$ is the so-called cylindrical $\alpha$-stable process. In this case, $\int_{|x|>1} |x|^\theta\mu(dz)<\infty$ holds for any $\theta\in(0,\alpha)$. 


\textbf{Example 2: Subordinated cylindrical L\'{e}vy process.}  For $H=L^2(D)$, let $(W_t)_{t\geq0}$ be a $Q$-Wiener process on $H$, where $Q$ is a nonnegative symmetric bounded linear operator on $H$ with finite trace and non-degenerate, i.e., $Ker Q=\{0\}$. For $\bar{\alpha}\in(0,2)$, Let $\{S_t\}_{t\geq0}$ be an $\bar{\alpha}/2$-stable subordinator, which is independent of $(W_t)_{t\geq0}$. Now we define
$$L^{\bar{\alpha}}_t:=W_{S_t},\,\forall t\geq0,$$
then $\nu$, the intensity measure of $L^{\bar{\alpha}}_t$ on $H$,  satisfies conditions (\textbf{C1}) and (\textbf{C2}) with any $\alpha\in(0,\bar{\alpha})$.

\textbf{Example 3: Compound Poisson process.} As a special case of cylindrical L$\rm\acute{e}$vy process, we construct an example which is a compound Poisson process. In \textbf{Example 1}, we assume that $\mu$, the intensity measure of $\{L_i(t)\}_{i\in\mathbb{N}}$ on $(\mathbb{R},\mathcal{B}(\mathbb{R}))$, satisfies that
$$\mu(\{1\})=\mu(\{-\sqrt{2}\})=1;\;\mu(\mathbb{R}\backslash\{1,-\sqrt{2}\})=0,$$
then $(L(t))_{t\geq0}$ becomes a compound Poisson process, and intensity measure of $L(t)$ satisfies conditions (\textbf{C1}) and (\textbf{C2}) with any $\alpha\in(0,2]$.

\section{Existence and Uniqueness of solutions}\label{240418.1347}
\subsection{Yosida-type approximation}

In this section, we will prove the existence and uniqueness of strong solutions to equation (\ref{2.3}) and establish the Markov property. To this end, we will construct a sequence of approximating solutions via Yosida approximations. First we list some elementary properties of the function $\beta$ below, which are easy to prove, so we omit the details.
\begin{lem}\label{[3.1]}  The function $\beta$ satisfies:\\
{\rm (i)} $\beta(\mathbb{R})=\mathbb{R}$;\\
{\rm (ii)} There exists a constant $K>0$ such that $|\beta(r)-\beta(s)|\leq K|r-s|$, $\forall\, r,s\in\mathbb{R}$;\\
{\rm (iii)} $\left(\beta(r)-\beta(s)\right)(r-s)\geq\frac{1}{K}\left(\beta(r)-\beta(s)\right)^2\geq0$, $\forall\, r,s\in\mathbb{R}$;\\
{\rm (iv)} There exist $c_1,c_2>0$ such that $r\beta(r)\geq c_1r^2-c_2$, $\forall\, r\in\mathbb{R}$.
\end{lem}

\indent Let $I$ be the identity map. For $\epsilon>0$, the function $\beta+\epsilon I:\mathbb{R}\to\mathbb{R}$ is a bijective, its inverse function $(\beta+\epsilon I)^{-1}$ has the following expression:
\begin{equation}\label{3.1}
(\beta+\epsilon I)^{-1}r=
\begin{cases} \cfrac{r}{a+\epsilon},\,r<0;\\ \cfrac{r}{\epsilon},\, r\in[0,\epsilon\rho];\\
\cfrac{r+\rho}{(1+\epsilon)},\,r>\epsilon\rho.\end{cases}
\end{equation}
\begin{rmk}\label{[3.2]} $\rm (i)$ $\beta+\epsilon I$ is Lipschitz continuous and strictly monotone. $(\beta+\epsilon I)^{-1}$ has bounded weak derivative and $(\beta+\epsilon I)^{-1}(0)=0$, hence $(\beta+\epsilon I)^{-1}x\in H_0^1(D)$ for any $x\in H_0^1(D)$. $\beta+\epsilon I$ and $(\beta+\epsilon I)^{-1}$ can be seen as operators from $H_0^1(D)$ to $H_0^1(D)$.\\
$\rm (ii)$ It is easy to see that there exists a constant $C>0$ which is independent of $\epsilon$ such that
$$|(\beta+\epsilon I)^{-1}r|\leq C(1+|r|),\;\forall\, r\in\mathbb{R}.$$
\end{rmk}

\indent Now, for $x\in H_0^1(D)$ and $\epsilon>0$, let $G_\epsilon x:=-\Delta(\beta+\epsilon I)x$. Since $G_\epsilon$ is maximal monotone (see Section 1.2.4 in \cite{MR3560817} for definition) in $H^{-1}(D)$, we can define the operator $J_\epsilon$ and $F_\epsilon$ on $H^{-1}(D)$ as follows:
$$J_\epsilon y:=(I+\epsilon G_\epsilon)^{-1}y,\;
F_\epsilon y:=\frac{1}{\epsilon}(y-J_\epsilon y)=-\Delta(\beta+\epsilon I)J_\epsilon y,\;\forall\, y\in H^{-1}(D).$$
For convenience we denote $Z_\epsilon:=(\beta+\epsilon I)J_\epsilon$, then $F_\epsilon=-\Delta Z_\epsilon$. There are some important properties of those operators defined above, which have already been proved (see Lemma 2.3.1 and Lemma 2.3.2 in \cite{MR3560817}), and we list them below.

\begin{lem}\label{[3.3]} For $\epsilon>0$, we have:\\
$\rm(i)$ $||J_\epsilon x||_{-1}\leq ||x||_{-1}$ for $\forall x\in H^{-1}(D)$, $|J_\epsilon x|_2\leq |x|_2$ for $\forall x\in L^2(D)$;\\
$\rm(ii)$ for $\forall x_1,x_2\in H^{-1}(D)$, $y_1,y_2\in L^2(D)$, we have
\begin{equation}\label{3.2}||J_\epsilon x_1-J_\epsilon x_2||_{-1}\leq ||x_1-x_2||_{-1},\;
|J_\epsilon y_1-J_\epsilon y_2|_2\leq \frac{2}{\epsilon}|x_1-x_2|_2.\end{equation}
Therefore, $J_\epsilon$ and $F_\epsilon$ are Lipschitz continuous both on $H^{-1}(D)$ and on $L^2(D)$;\\
$\rm(iii)$ $J_\epsilon x\in H_0^1(D)$ and $Z_\epsilon x\in H_0^1(D)$ for $x\in H^{-1}(D)$. Then $J_\epsilon$ is an operator from $H^{-1}(D)$ to $H_0^1(D)$.\\
$\rm(iv)$ For $x\in H^{-1}(D)$, we have
\begin{equation}\label{3.3} \left<F_\epsilon x,x\right>_{H^{-1}}=\left<Z_\epsilon x,J_\epsilon x\right>_{L^2}+\epsilon ||F_\epsilon x||_{-1}^2,\end{equation}
and when $x\in L^2(D)$,
\begin{equation}\label{3.4} \left<F_\epsilon x,x\right>_{L^2}=\left<-\Delta Z_\epsilon x,J_\epsilon x\right>_{L^2}+\epsilon|F_\epsilon x|_{2}^2.\end{equation}
\end{lem}

\begin{lem}\label{[3.4]} There exist $\gamma,\epsilon_0>0$ such that for any $\epsilon\in(0,\epsilon_0)$ and $x\in L^2(D)$,
\begin{equation}\label{3.5}\left<F_\epsilon x,x\right>_{L^2}\geq\gamma||Z_\epsilon x||_1^2+\epsilon|F_\epsilon x|_{2}^2.\end{equation}
\end{lem}
\begin{proof}Since $J_\epsilon=(\beta +\epsilon I)^{-1}Z_\epsilon$, by (\ref{3.1}),(\ref{3.4}) and integration by part we have
\begin{align}\label{3.6} \left<F_\epsilon x,x\right>_{L^2}&=
\left<\nabla Z_\epsilon x,\nabla(\beta +\epsilon I)^{-1}Z_\epsilon x\right>_{L^2}+\epsilon|F_\epsilon x|_{2}^2\notag\\
&\geq \min\{(a+\epsilon)^{-1},\epsilon^{-1},(1+\epsilon)^{-1}\}||Z_\epsilon x||_1^2+\epsilon|F_\epsilon x|_{2}^2\notag\\
&:=\gamma_\epsilon||Z_\epsilon x||_1^2+\epsilon|F_\epsilon x|_{2}^2\geq0.
\end{align}
Here we choose $\epsilon_0<\min\{a,1\}$, then for $\epsilon\in(0,\epsilon_0)$,
$$\epsilon^{-1}\geq\max\{a^{-1},1\}\geq\min\{(a+\epsilon)^{-1},(1+\epsilon)^{-1}\}.$$
Letting $\gamma=\min\{(a+\epsilon_0)^{-1},(1+\epsilon_0)^{-1}\}$, we have $\gamma_\epsilon\geq\gamma$ for any $\epsilon\in(0,\epsilon_0)$. The proof is complete.
\end{proof}

\begin{rmk}\label{[3.5]} $\rm (\ref{3.5})$ plays an important role in proving the existence of strong solutions. We like to emphasis that when using the integration by part, the boundary term vanishes because of Lemma \ref{[3.3]} (iii). This is the main reason we choose $F_\epsilon$ as the approximation operator of $A$. We call this approximation the Yosida-type approximation, since the Yosida approximation is for the operator $G_\epsilon$ instead of original operator $A$. This is the main difference from the Yosida approximation used in \cite{MR1942322}. This Yosida-type approximation is used to overcome the difficulties caused by the multivalues in the definition of $(I+\epsilon A)^{-1}$. In the rest of this paper we always assume $\epsilon\in(0,\epsilon_0)$.
\end{rmk}

\subsection{Approximating equations}

With the preparations above, we consider the following approximating equations: for $\epsilon>0$,
\begin{equation}\label{3.7}\begin{cases} dX_\epsilon(t)=\epsilon\Delta X_\epsilon(t)-F_\epsilon X_\epsilon(t)+dL_t,\\X_\epsilon(0)=x.\end{cases}\end{equation}
We also need the following equations driven by $\tilde{L}_t$ (see (\ref{240416.2003})), the part of ``small jumps'' of $L_t$:
\begin{equation}\label{3.8}\begin{cases} dY_\epsilon(t)=\epsilon\Delta Y_\epsilon(t)-F_\epsilon Y_\epsilon(t)+d\tilde{L}_t,\\Y_\epsilon(0)=x.\end{cases}\end{equation}
Here we add the term ``$\epsilon\Delta X_\epsilon(t)$'' in (\ref{3.7}) and (\ref{3.8}) since it will later allow us to prove the existence of invariant measures. From Lemma \ref{[3.3]} (ii), we know that $F_\epsilon$ is Lipschitz continuous in $L^2(D)$, then it is easy to prove that for $x\in L^2(D)$, both (\ref{3.7}) and (\ref{3.8}) admit a unique variational solution. Here we omit the proof and refer the readers to \cite{MR3158475} for details.

\begin{lem}\label{[3.6]} For $\epsilon\in(0,\epsilon_0)$ and $x\in L^2(D)$, the solution to (\ref{3.7}) (or (\ref{3.8})) exists uniquely. Furthermore, there is a positive constant $C_T$ such that for  any $x_1,x_2\in L^2(D)$, we have
\begin{equation}\label{v1.2 eq0} \sup_{t\in[0,T]}|X_\epsilon(t,x_1)-X_\epsilon(t,x_2)|_2^2\leq C_T |x_1-x_2|_2^2,
\end{equation}
where $X_\epsilon(t,x_1)$ and $X_\epsilon(t,x_2)$ are the solutions to (\ref{3.7}) with initial values $x_1$ and $x_2$ respectively. 
\end{lem}

\begin{rmk}\label{[3.7]} In fact, when $x\in H^{-1}(D)$, (\ref{3.7}) also has a unique solution $X_\epsilon$ which is an $H^{-1}(D)$-valued $\rm c\grave{a}dl\grave{a}g$  $\mathbb{F}$-adapted Markov process. In that case, we may consider another Gelfand triple $L^2(D)\subset H^{-1}(D)\subset (L^2(D))^*$. This holds similarly for (\ref{3.8}).
\end{rmk}

\indent To prove Theorem \ref{[2.3]}, we first prove the existence of a strong solution by the Yosida-type approximations, and then we show the uniqueness.

\begin{prop}\label{[3.8]} For $x\in L^2(D)$, equation (\ref{2.3}) has a strong solution.
\end{prop}
\begin{proof}By means of the standard interlacing procedure (see \cite{MR3158475} or \cite{MR1011252} for details), it suffices to consider the small jumps and show that the following equation has a unique strong solution:
\begin{equation}\label{3.10}\begin{cases} dY(t)=-AY(t)dt+d\tilde{L}_t,\\Y(0)=x\in L^2(D).\end{cases}\end{equation}

Let $Y_\epsilon=Y_\epsilon(t)$ be the solution to equation (\ref{3.8}). The proof will be divided into three steps. In the first step we will give a priori estimates for $Y_\epsilon$. In the second step we will show that $\{Y_\epsilon\}_{\epsilon>0}$ is a Cauchy sequence in $L^2(\Omega; D([0,T];H^{-1}(D)))$, where $D([0,T];H^{-1}(D))$ is the space of all $H^{-1}(D)$-valued $\rm c\grave{a}dl\grave{a}g$  functions endowed with uniform norm. In the final step, we will take the limits to obtain a strong solution to equation (\ref{3.10}).\\

\textbf{Step1}: \textbf{A priori estimates}. Applying It$\rm\hat{o}$ formula to $|Y_\epsilon(t)|_2^2$ we have
\begin{align}\label{3.11} |Y_\epsilon(t)|_2^2=&|x|_2^2-2\epsilon\int_0^t||Y_\epsilon(s)||_1^2ds-2\int_0^t\left<F_\epsilon Y_\epsilon(s),Y_\epsilon(s)\right>_{L^2}ds\notag\\
+&2\int_0^t\int_{|z|_2\leq1}\left<Y^\epsilon(s-),z\right>_{L^2}\tilde{N}(dsdz)+\int_0^t\int_{|z|_2\leq1}|z|_2^2 N(dsdz).
\end{align}
By (\ref{3.5}), taking the supremum over time and then taking expectations on both side of (\ref{3.11}) we find
\begin{align}\label{3.12} &\mathbb{E}\left[\sup_{t\in[0,T]}|Y_\epsilon(s)|_2^2\right]+2\epsilon\mathbb{E}\int_0^t||Y_\epsilon(s)||_1^2ds+2\gamma\mathbb{E}\int_0^T||Z_\epsilon Y_\epsilon(s)||_1^2ds+2\epsilon\mathbb{E}\int_0^T|F_\epsilon Y_\epsilon(s)|_2^2ds\notag\\
\leq &|x|_2^2+2\mathbb{E}\sup_{t\in[0,T]}\Big|\int_0^t\int_{|z|_2\leq1}\left<Y_\epsilon(s-),z\right>_{L^2}\tilde{N}(dsdz)\Big|+\mathbb{E}\int_0^T\int_{|z|_2\leq1}|z|_2^2 N(dsdz).
\end{align}
The last term in (\ref{3.12}) is finite since
\begin{equation}\label{3.13}\mathbb{E}\int_0^T\int_{|z|_2\leq1}|z|_2^2 N(dsdz)=T\int_{|z|_2\leq1}|z|_2^2\nu(dz)<\infty.\end{equation}
By the BDG inequality we have
\begin{align}\label{3.14} &\mathbb{E}\sup_{t\in[0,T]}\Big|\int_0^t\int_{|z|_2\leq1}\left<Y_\epsilon(s-),z\right>_{L^2}\tilde{N}(dsdz)\Big|\notag\\
\leq&C\left(\mathbb{E}\int_0^T\int_{|z|_2\leq1}|Y_\epsilon(s-)|_2^2\,|z|_2^2\nu(dz)ds\right)^{\frac{1}{2}}\notag\\
\leq&C\left(\mathbb{E}\int_0^T|Y_\epsilon(s-)|_2^2ds\right)^{\frac{1}{2}}
\leq C\left(1+\int_0^T\mathbb{E}[\sup_{t\in[0,s]}|Y_\epsilon(t)|_2^2]ds\right).
\end{align}
Now, combining (\ref{3.12})$\sim$(\ref{3.14}) together and using the Gronwall inequality yields
\begin{align}\label{3.15} &\mathbb{E}[\sup_{t\in[0,T]}|Y_\epsilon(t)|_2^2]+2\epsilon\mathbb{E}\int_0^T||Y_\epsilon(s)||_1^2ds\notag\\
&+2\gamma\mathbb{E}\int_0^T||Z_\epsilon Y_\epsilon(s)||_1^2ds+2\epsilon\mathbb{E}\int_0^T|F_\epsilon Y_\epsilon(s)|_2^2ds\leq C,
\end{align}
where the constant $C$ is independent of $\epsilon$.\\

\textbf{Step2}: \textbf{Cauchy convergence}. We will show that $\{Y_\epsilon\}_{\epsilon\in(0,\epsilon_0)}$ is  a Cauchy sequence in $L^2(\Omega; D([0,T];H^{-1}(D)))$. For $\epsilon,\lambda\in(0,\epsilon_0)$, by applying the It$\rm\hat{o}$ formula to $||Y^\epsilon(t)-Y^\lambda(t)||_{-1}^2$ we have
\begin{align}\label{3.16} ||Y_\epsilon(t)-Y_\lambda(t)||_{-1}^2=&-2\int_0^t\left<\epsilon Y_\epsilon(s)-\lambda Y_\lambda(s),Y_\epsilon(s)-Y_\lambda(s)\right>_{L^2}ds\notag\\
-&2\int_0^t\left<F_\epsilon Y_\epsilon(s)-F_\lambda Y_\lambda(s),Y_\epsilon(s)-Y_\lambda(s)\right>_{H^{-1}}ds.
\end{align}
On the one hand,
\begin{align}\label{3.17} &-2\left<\epsilon Y_\epsilon(s)-\lambda Y_\lambda(s),Y_\epsilon(s)-Y_\lambda(s)\right>_{L^2}\notag\\
\leq& 2(\epsilon+\lambda)\left<Y_\epsilon(s),Y_\lambda(s)\right>_{L^2}\notag\\
\leq& (\epsilon+\lambda)(|Y_\epsilon(s)|_2^2+|Y_\lambda(s)|_2^2).
\end{align}
On the other hand, by the definition of $F_\epsilon$,
\begin{align}\label{3.18}&\left<F_\epsilon Y_\epsilon(s)-F_\lambda Y_\lambda(s),Y_\epsilon(s)-Y_\lambda(s)\right>_{H^{-1}}\notag\\
=&\left<(\beta+\epsilon I)J_\epsilon Y_\epsilon(s)-(\beta+\epsilon I)J_\lambda Y_\lambda(s),
J_\epsilon Y_\epsilon(s)-J_\lambda Y_\lambda(s)\right>_{L^2}\notag\\
&+(\epsilon-\lambda)\left<J_\lambda Y_\lambda(s),J_\epsilon Y_\epsilon(s)-J_\lambda Y_\lambda(s)\right>_{L^2}\notag\\
&+\left<F_\epsilon Y_\epsilon(s)-F_\lambda Y_\lambda(s),\epsilon F_\epsilon Y_\epsilon(s)-\lambda F_\lambda Y_\lambda(s)\right>_{H^{-1}}.
\end{align}
By Lemma \ref{[3.1]} we have
\begin{align}\label{3.19}&\left<(\beta+\epsilon I)J_\epsilon Y_\epsilon(s)-(\beta+\epsilon I)J_\lambda Y_\lambda(s),
J_\epsilon Y_\epsilon(s)-J_\lambda Y_\lambda(s)\right>_{L^2}\notag\\
\geq&\frac{1}{K}|\beta(J_\epsilon Y_\epsilon(s))-\beta(J_\lambda Y_\lambda(s))|_2^2
+\epsilon|J_\epsilon Y_\epsilon(s)-J_\lambda Y_\lambda(s)|_2^2.
\end{align}
By Lemma \ref{[3.3]} (i) we have
\begin{align}\label{3.20} &|(\epsilon-\lambda)\left<J_\lambda Y_\lambda(s),J_\epsilon Y_\epsilon(s)-J_\lambda Y_\lambda(s)\right>_{L^2}|\notag\\
\leq&(\epsilon+\lambda)|J_\lambda Y_\lambda(s)|_2(|J_\epsilon Y_\epsilon(s)|_2+|J_\lambda Y_\lambda(s)|_2)\notag\\
\leq&(\epsilon+\lambda)|Y_\lambda(s)|_2(|Y_\epsilon(s)|_2+|Y_\lambda(s)|_2).
\end{align}
Similar to (\ref{3.17}), we have
\begin{align}\label{3.21} &\left<F_\epsilon Y_\epsilon(s)-F_\lambda Y_\lambda(s),\epsilon F_\epsilon Y_\epsilon(s)-\lambda F_\lambda Y_\lambda(s)\right>_{H^{-1}}\notag\\
\geq&-(\epsilon+\lambda)\left<F_\epsilon Y_\epsilon(s),F_\lambda Y_\lambda(s)\right>_{H^{-1}}\notag\\
\geq&-\frac{1}{2}(\epsilon+\lambda)(||F_\epsilon Y_\epsilon(s)||_{-1}^2+||F_\lambda Y_\lambda(s)||_{-1}^2).
\end{align}

Now, combining (\ref{3.16})$\sim$(\ref{3.21}) together yields
\begin{align}\label{3.22} &||Y_\epsilon(t)-Y_\lambda(t)||_{-1}^2+\int_0^t|\beta(J_\epsilon Y_\epsilon(s))-\beta(J_\lambda Y_\lambda(s))|_2^2ds\notag\\
\leq& C(\epsilon+\lambda)\int_0^t \Big[ |Y_\epsilon(s)|_2^2+|Y_\lambda(s)|_2^2+||F_\epsilon Y_\epsilon(s)||_{-1}^2+||F_\lambda Y_\lambda(s)||_{-1}^2 \Big] ds.
\end{align}
Noticing that $||F_\epsilon Y_\epsilon(s)||_{-1}=||Z_\epsilon Y_\epsilon(s)||_1$,  it follows from (\ref{3.15}) and (\ref{3.22}) that
\begin{equation}\label{3.23}\mathbb{E}[\sup_{t\in[0,T]}||Y_\epsilon(t)-Y_\lambda(t)||_{-1}^2]
+\mathbb{E}\int_0^T|\beta(J_\epsilon Y_\epsilon(s))-\beta(J_\lambda Y_\lambda(s))|_2^2ds\leq C(\epsilon+\lambda),
\end{equation}
where the constant $C$ is independent of $\epsilon$ and $\lambda$.

This proves that $\{Y_\epsilon \}_{\epsilon\in(0,\epsilon_0)}$ is a Cauchy sequence in the space $L^2(\Omega;D([0,T];H^{-1}(D)))$ and $\{\beta(J_\epsilon Y_\epsilon)\}_{\epsilon\in(0,\epsilon_0)}$ is a Cauchy sequence in $L^2(\Omega\times[0,T];L^2(D))$. Hence there exist processes $Y\in L^2(\Omega;D([0,T];H^{-1}(D)))$ and $W\in L^2(\Omega\times[0,T];L^2(D))$ such that
\begin{equation}\label{3.25}\lim_{\epsilon\to0}\mathbb{E}[\sup_{t\in[0,T]}||Y_\epsilon(t)-Y(t)||_{-1}^2]=0.
\end{equation}
\begin{equation}\label{3.24}\lim_{\epsilon\to0}\mathbb{E}\int_0^T|\beta(J_\epsilon Y_\epsilon(s))-W(s)|_2^2ds=0.
\end{equation}

\textbf{Step3}: \textbf{Taking limits}. From (\ref{3.15}) we have $\mathbb{E}[\sup_{t\in[0,T]}|Y_\epsilon(t)|_2^2]<C$,
where $C$ is independent of $\epsilon$. Hence by the Alaoglu theorem and the uniqueness of limit, there exists a sequence $\{\epsilon_k\}_{k\in\mathbb{N}}$ such that $\epsilon_k\to0$ and
$Y_{\epsilon_k}\stackrel{w^*}{\rightarrow} Y$ in the dual space of $L^2(\Omega;L^1([0,T];L^2(D)))$, where ``$\stackrel{w^*}{\rightarrow}$" means the weak-star convergence, and $Y\in L^2(\Omega,L^{\infty}([0,T];L^2(D)))$.
Moreover, by Fatou's Lemma we can see that $\sup_{t\in[0,T]}\mathbb{E}|Y(t)|_2^2\leq C$.

Keeping in mind that $I-J_\epsilon=\epsilon F_\epsilon$, by (\ref{3.15}) and Lemma \ref{[3.3]} (i), it follows that
\begin{equation}\label{3.26}\mathbb{E}\int_0^T||Y_\epsilon(s)-J_\epsilon Y_\epsilon(s)||_{-1}^2ds
\leq \epsilon^2\mathbb{E}\int_0^T||F_\epsilon Y_\epsilon||_{-1}^2ds\leq C\epsilon^2,
\end{equation}
\begin{equation}\label{3.27}\mathbb{E}\int_0^T|J_\epsilon Y_\epsilon(s)|^2_2ds
\leq T\mathbb{E}\sup_{t\in[0,T]}|Y_\epsilon(t)|_2^2\leq C.
\end{equation}
By (\ref{3.25}) and uniqueness of limit, we deduce that there exists a sequence (still denote by $\epsilon_k$) such that when $k\to\infty$, $J_{\epsilon_k} Y_{\epsilon_k}\to Y$ in $L^2(\Omega\times[0,T];H^{-1}(D))$ and $J_{\epsilon_k} Y_{\epsilon_k}\stackrel{w}{\rightarrow} Y$ in $L^2(\Omega\times[0,T];L^2(D))$, where ``$\stackrel{w}{\rightarrow}$" means the weak convergence.

Moreover, due to (\ref{3.24}) and weak convergence of $J_{\epsilon_k} Y_{\epsilon_k}$ we have
\begin{equation}\label{3.28} \lim_{k\to\infty} \mathbb{E}\int_0^T\left<\beta(J_{\epsilon_k} Y_{\epsilon_k}(s)),J_{\epsilon_k} Y_{\epsilon_k}(s)\right>_{L^2}ds
=\mathbb{E}\int_0^T\left<W(s),Y(s)\right>_{L^2}ds,
\end{equation}
Now, one can follow the same arguements as in the proof of Proposition 1.2.9 in \cite{MR3560817} to conclude that $W=\beta(Y)$ for a.e. $(\omega,t;x)\in\Omega\times[0,T]\times D$.

To show $Y$ is a strong solution to (\ref{3.10}), it remains to prove that $\beta(Y)\in L^2(\Omega\times[0,T];H_0^1(D))$. By (\ref{3.15}) we have
$$\mathbb{E}\int_0^T||Z_\epsilon Y_\epsilon||_{1}^2ds\leq C.$$
So by uniqueness of limit there exists a subsequence (still denoted as $\epsilon_k$) such that when $k\to\infty$, $Z_{\epsilon_k}Y_{\epsilon_k}=(\beta+\epsilon_k I)J_{\epsilon_k} Y_{\epsilon_k}\stackrel{w}{\rightarrow}W=\beta(Y)$ in $L^2(\Omega\times[0,T];H_0^1(D))$.
Taking the limit in (\ref{3.8}), we see that $Y$ is a strong solution to (\ref{3.10}). The proof is complete.
\end{proof}

\begin{prop}\label{[3.9]} For $x\in L^2(D)$, the strong solution to (\ref{2.3}) is unique.
\end{prop}
\begin{proof} For any $x_1,x_2\in L^2(D)$, let $X(t,x_1)$ and $X(t,x_2)$ be the strong solutions to (\ref{2.3}) with initial data $x_1$ and $x_2$ respectively. For $t\in[0,T]$, applying the chain rule to $||X_1(t)-X_2(t)||_{-1}^2$ we have
\begin{align}\label{3.29} &||X(t,x_1)-X(t,x_2)||_{-1}^2\notag\\
=&||x_1-x_2||^2_{-1}-2\int_0^t\left<\beta(X(s,x_1))-\beta(X(s,x_2)),X(s,x_1)-X(s,x_2)\right>_{L^2}ds\notag\\
\leq& ||x_1-x_2||^2_{-1},\quad\forall\, t\in[0,T],
\end{align}
where we have used Lemma \ref{[3.1]} (iii) in last inequality. Then the uniqueness of strong solutions follows from (\ref{3.29}) if we take $x_1=x_2$.
\end{proof}

\indent We have proved that there exists a unique strong solution $X$ to equation (\ref{2.3}). To complete the proof of Theorem \ref{[2.3]} we also need to prove the Markov property of $X$.

\subsection{Markov property in $L^2(D)$}

\indent Let $X=(X(t,x))_{t\geq0}$ be the strong solution to (\ref{2.3}) with initial data $x\in L^2(D)$, and $(Q_t)_{t\geq0}$ be the transition semigroup of $X$. To prove the Markov property of $X$, we may start with the approximate solution $X_\epsilon$. By Remark \ref{[3.7]}, we denote by $X_\epsilon=(X_\epsilon(t,x))_{t\geq0}$ the solution to equation (\ref{3.7}) with initial data $x\in H^{-1}(D)$, and by $(Q^\epsilon_t)_{t\geq0}$ the transition semigroup of $X_\epsilon$, namely for any $\varphi\in C_b(H^{-1}(D))$, $x\in H^{-1}(D)$ and $t>0$,  $Q^\epsilon_t\varphi(x)=\mathbb{E}[\varphi(X_\epsilon(t,x)]$. Next we will provide two lemmas regarding the relationship between $Q_t$ and $Q^\epsilon_t$, which will be used to prove the Markov property of $X$.

\begin{lem}\label{[4.11]} For any bounded subset $V\subset L^2(D)$ and any fixed $\delta,t>0$,
\begin{equation}\label{4.24} \lim_{\epsilon\to0}\sup_{x\in V}\mathbb{P}\left(||X_\epsilon(t,x)-X(t,x)||_{-1}>\delta\right)=0.
\end{equation}
\end{lem}

\begin{proof} For $M\in\mathbb{N}$, let $\tau_M=\inf\{t\geq0:\,N(\{|z|_2> M\}\times[0,t])=1\}$. Consider the following two equations with $x\in L^2(D)$:
\begin{align}\label{4.25}X_{\epsilon,M}(t,x)=&x+\epsilon\int_0^t\Delta X_{\epsilon,M}(s)ds-\int_0^t F_\epsilon X_{\epsilon,M}(s)ds\notag\\
&+\int_0^t\int_{|z|_2\leq1}z\tilde{N}(dzds)+\int_0^t\int_{1<|z|_2\leq M}z N(dzds),
\end{align}
\begin{equation}\label{4.26}
X_{M}(t,x)=x-\int_0^t AX_{M}(s)ds+\int_0^t\int_{|z|_2\leq1}z\tilde{N}(dzds)+\int_0^t\int_{1<|z|_2\leq M}z N(dzds).
\end{equation}
Clearly on $\{\tau_M>t\}$ we have $X_\epsilon(t,x)=X_{\epsilon,M}(t,x)$, $X(t)=X_M(t,x)$, $\mathbb{P}$-a.s.. On the other hand,
\begin{align}\label{4.27} &\mathbb{P}\left(||X_\epsilon(t,x)-X(t,x)||_{-1}>\delta\right)\notag\\
=&\mathbb{P}\left(||X_\epsilon(t,x)-X(t,x)||_{-1}>\delta,\,\tau_M\leq t\right)\notag\\
+&\mathbb{P}\left(||X_\epsilon(t,x)-X(t,x)||_{-1}>\delta,\,\tau_M> t\right)\notag\\
\leq&\mathbb{P}\left(\tau_M\leq t\right)+\frac{1}{\delta^2}\mathbb{E}||X_{\epsilon,M}(t,x)-X_M (t,x)||_{-1}^2.
\end{align}
Using the similar arguements as in the proof of Proposition \ref{[3.8]} (especially the proof of (\ref{3.23})), one can prove that there exists a constant $C(x,M)>0$ such that
\begin{equation}\label{4.28} \mathbb{E}[\sup_{t\in[0,T]}||X_{\epsilon,M}(t,x)-X_M (t,x)||_{-1}^2]\leq C(x,M)\epsilon,
\end{equation}
and one can show that $C(x,M)$ has an upper bound $C(M)$ over the bounded subset $V\subset L^2(D)$. Thus we have, for any $M>0$,
\begin{equation}\label{4.29} \lim_{\epsilon\to0}\sup_{x\in V}\mathbb{E}||X_{\epsilon,M}(t,x)-X_M (t,x)||_{-1}^2=0.
\end{equation}
Meanwhile, since $\nu(\{|z|_2>1\})<\infty$, by Chebyshev's inequality we have
\begin{align}\label{4.30} \mathbb{P}\left(\tau_M\leq t\right)&=\mathbb{P}\left(N(\{|z|_2\geq M\}\times[0,t])\geq1\right)\notag\\
&\leq\mathbb{E}N(\{|z|_2\geq M\}\times[0,t])\notag\\
&=\nu(\{|z|_2\geq M\})t\to0.\;(M\to\infty)
\end{align}
From (\ref{4.27}), (\ref{4.29}) and (\ref{4.30}) we deduce (\ref{4.24}).
\end{proof}
\begin{lem}\label{[4.12]} For any bounded subset $V\subset L^2(D)$ and any $\phi\in {\rm Lip}_b(H^{-1}(D))$,
\begin{equation}\label{4.31} \lim_{\epsilon\to0}\sup_{x\in V}\Big|Q^\epsilon_t\phi(x)-Q_t\phi(x)\Big|=0.
\end{equation}
\end{lem}
\begin{proof} Since $\phi\in {\rm Lip}_b(H^{-1}(D))$, we have for some $\delta>0$,

\begin{align} \label{4.32} \Big|Q^\epsilon_t\phi(x)-Q_t\phi(x)\Big|
\leq&\mathbb{E}\Big|\phi(X_\epsilon(t,x))-\phi(X(t,x))\Big|\notag\\
\leq& C_1\delta+2C_2\mathbb{P}\left(||X_\epsilon(t,x)-X(t,x)||_{-1}\geq \delta\right),
\end{align}
where $C_1$ is the Lipschitz constant of $\phi$ and $C_2$ is the upper bound of $|\phi|$. Now taking the supremum over $x\in V$ and letting $\epsilon\to0$ in (\ref{4.32}), by Lemma \ref{[4.11]} and the arbitrariness of $\delta>0$, we deduce (\ref{4.31}).
\end{proof}
\begin{prop}\label{[5.1]} The strong solution $X$ to (\ref{2.3}) is a Markov process, that is, for any $t,s\geq0$, $x\in L^2(D)$ and $G\in B_b(L^2(D))$,
\begin{equation}\label{5.1} \mathbb{E}\left[G(X(t+s,x))|\mathcal{F}_s\right](\omega)=Q_t G(X(s,x)(\omega))
\end{equation}
for $\mathbb{P}$-a.s. $\omega\in\Omega$.
\end{prop}
\begin{proof} First, we know that for each $\epsilon>0$, the approximation solution $X_\epsilon$ to system (\ref{3.7})  is a Markov process in $L^2(D)$ (see Theorem 9.30 in \cite{MR2356959} for example), which implies that any $t,s\geq0$, $x\in L^2(D)$ and $F\in {\rm Lip}_b(H^{-1}(D))\subset C_b(L^2(D))$, we have
\begin{equation}\label{5.2} \mathbb{E}\left[F(X_\epsilon(t+s,x))|\mathcal{F}_s\right] =Q^\epsilon_t F(X_\epsilon(s,x))
\end{equation}
for $\mathbb{P}$-a.s. $\omega\in\Omega$. By Lemma \ref{[4.11]}, $||X_\epsilon(t,x)-X(t,x)||_{-1}$ converges to 0 in probability. By the dominated convergence theorem, we have
\begin{equation}\label{5.3} \lim_{\epsilon\to0}\mathbb{E}\Big|\mathbb{E}\left[F(X_\epsilon(t+s,x))|\mathcal{F}_s\right]-
\mathbb{E}\left[F(X(t+s,x))|\mathcal{F}_s\right]\Big|=0.
\end{equation}
Thus, as $\epsilon\to0$ we have
\begin{equation}\label{5.4}\mathbb{E}\left[F(X_{\epsilon}(t+s,x))|\mathcal{F}_s\right]
\stackrel{\mathbb{P}}{\rightarrow}\mathbb{E}\left[F(X(t+s,x))|\mathcal{F}_s\right],
\end{equation}
where ``$\stackrel{\mathbb{P}}{\rightarrow}$" means convergence in probability. Note that
\begin{align} \label{5.5} &Q ^{\epsilon}_t F(X_{\epsilon}(s,x))-Q_t F(X(s,x))\notag\\
=&\left[Q^{\epsilon}_t F(X_{\epsilon}(s,x))-Q^{\epsilon}_t F(X(s,x))\right]\notag\\
+&\left[Q^{\epsilon}_t F(X(s,x))-Q_t F(X(s,x))\right]\notag\\
:=&I_1+I_2.
\end{align}
On the one hand, for any $x,y\in L^2(D)$ and $t>0$, we can prove that
$$||X_\epsilon(s,x)-X_\epsilon(s,y)||^2_{-1}\leq e^{-2\epsilon t}||x-y||^2_{-1},$$
see (\ref{4.20}) for the proof. Then $Q^\epsilon_t F: H^{-1}(D)\to\mathbb{R}$ is Lipschitz continuous for any $F\in {\rm Lip}_b(H^{-1}(D))$ and the Lipschitz coefficients are uniformly bounded with respect to $\epsilon$. Thus, similar to (\ref{5.3}), we have $|I_1|\stackrel{\mathbb{P}}{\rightarrow}0$ as $\epsilon\to0$. On the other hand, when $x\in L^2(D)$, the strong solution $X(s,x)\in L^2(D)$, $\mathbb{P}$-a.s. Thus by Lemma \ref{[4.12]} we have $|I_2|\stackrel{\mathbb{P}}{\rightarrow}0$ as $\epsilon\to0$. Now, from (\ref{5.5}) we conclude that as $\epsilon\to0$,
\begin{equation}\label{5.6}Q^\epsilon_t F(X_\epsilon(s,x))\stackrel{\mathbb{P}}{\rightarrow}Q_t F(X(s,x)).
\end{equation}
From (\ref{5.2}),(\ref{5.4}) and (\ref{5.6}) it follows that
\begin{equation}\label{5.7} \mathbb{E}\left[F(X(t+s,x))|\mathcal{F}_s\right]=Q_t F(X(s,x)),\; \mathbb{P}{\rm -a.s.}
\end{equation}
Since ${\rm Lip}_b(H^{-1}(D))$ is dense in $B_b(H^{-1}(D))$, it follows that (\ref{5.7}) also holds true for any $B_b(H^{-1}(D))$. We also notice that any function $F\in B_b(L^2(D))$ can be extended to a function $\tilde{F}\in B_b(H^{-1}(D))$ by letting
\begin{equation}\tilde{F}(x)=\begin{cases} F(x),\,x\in L^2(D),\\0,\,x\in H^{-1}(D)\backslash L^2(D).
\end{cases}\nonumber\end{equation}
Together with the fact that $X(t)\in L^2(D)$, $\mathbb{P}$-a.s. for any fixed $t>0$, we conclude that (\ref{5.7}) also holds for any $F\in B_b(L^2(D))$. The proof is complete.
\end{proof}

\section{Ergodicity}

\indent {We have now completed the proof of Theorem \ref{[2.3]}. In this section, we will prove the existence and uniqueness of the invariant measures of $(Q_t)_{t\geq0}$. Unfortunately, $(Q_t)_{t\geq0}$ may not be a Feller semigroup since we don't have the initial continuity of strong solution $X$ in $L^2(D)$. This prevent us from applying the usual Feller method when prove the uniqueness of invariant measures. To overcome this difficulty, we will take the approach of ``e-property'' to prove the uniqueness of invariant measures of $(Q_t)_{t\geq0}$. The key is to show the irreducibility of the transition semigroup of $X$ in $H^{-1}(D)$. To this end, we will use the criterion established in \cite{2207.11488}. For the existence of invariant measures of $(Q_t)_{t\geq0}$, it is difficult to construct a compact subset of $L^2(D)$ when using the Krylov-Bogoliubov criterion in $L^2(D)$, since $\beta(X(t))$ instead of $X(t)$ is in $H_0^1(D)$ for a.e. $t$. To overcome this difficulty, 
we will employ the Yosida-type approximations introduced in Section \ref{240418.1347} to construct an invariant measure of $(Q_t)_{t\geq0}$ by taking weak convergence limit of the invariant measures of approximating equations.}

\subsection{Uniqueness of invariant measure}
\indent In this subsection we will prove the uniqueness of invariant measures. First, we define the generalized solution to stochastic Stefan problem (see also equation (\ref{2.3}))
\begin{equation*}\begin{cases}dX(t)=\Delta\beta(X(t))dt+dL_t,t\in[0,T],\\ X(0)=x.
\end{cases}\end{equation*}
We will see that the transition semigroup of the generalized solution possesses some good properties such as Feller property, regular support, etc.
\begin{defn}\label{[4.1]} For $\tilde{x}\in H^{-1}(D)$, we say a $H^{-1}(D)$-valued stochastic process $\tilde{X}=(\tilde{X}(t,\tilde{x}))_{t\in[0,T]}$ the generalized solution to equation (\ref{2.3}) if there exists a sequence $(x_k)_{k\in\mathbb{N}}\subset L^2(D)$ such that as $k\to\infty$, $x_k\to \tilde{x}$ in $H^{-1}(D)$ and
$$\sup_{t\in[0,T]}||X_k(t,x_k)-\tilde{X}(t,\tilde{x})||_{-1}^2\to0,\;\mathbb{P}{\rm-a.s.,}$$
where for each k, $X_k(t,x_k)$ is the strong solution with initial value $x_k$.
\end{defn}

\begin{rmk}\label{[4.2]}$\rm(i)$ By (\ref{3.29}), for any $x_1,x_2\in L^2(D)$ and $T>0$, we have
\begin{equation}\label{4.1} \sup_{t\in[0,T]}||X(t,x_1)-X(t,x_2)||_{-1}^2\leq ||x_1-x_2||^2_{-1},
\end{equation}
which allows us to define the generalized solution of (\ref{2.3}) by continuously extending the map $x\mapsto X(\cdot,x)$ from $L^2(D)$ to $H^{-1}(D)$. Hence the generalized solutions to (\ref{2.3}) exist. 
Let $\tilde{X}_1$ and $\tilde{X}_2$ be the solutions with initial values $\tilde{x}_1$ and $\tilde{x}_2\in H^{-1}(D)$ respectively.  
By (\ref{4.1}), for any $T>0$, 
\begin{equation}\label{4.2} \sup_{t\in[0,T]}||\tilde{X}_1(t,\tilde{x}_1)-\tilde{X}_2(t,\tilde{x}_2)||_{-1}^2\leq ||\tilde{x}_1-\tilde{x}_2||^2_{-1},\;\mathbb{P}{\rm-a.s.,}
\end{equation}
which implies the uniqueness of generalized solutions. For $\tilde{x}\in L^2(D)$, $\tilde{X}$ coincides with $X$ the strong solution of (\ref{2.3}).

$\rm(ii)$ In fact, (\ref{2.3}) has variational solutions under the Gelfand triple $L^2(D)\subset H^{-1}(D)\subset (L^2(D))^*$, since $A=-\Delta\beta:L^2(D)\to(L^2(D))^*$(see Example 4.1.11 in \cite{MR3410409}), and we can verify that Hypotheses (H1)-(H4) in \cite{MR3158475} are satisfied. By (\ref{4.1}) we see that when $x\in L^2(D)$, the strong solution to (\ref{2.3}) is also a variational solution, then by the continuity of variational solutions with respect to initial values, and the definition of generalized solutions, we conclude that the generalized solution of (\ref{2.3}) are actually the variational solution to (\ref{2.3}) under the Gelfand triple $L^2(D)\subset H^{-1}(D)\subset (L^2(D))^*$.

$\rm(iii)$ By the property of the variational solution, $\tilde{X}$ is an $H^{-1}(D)$-valued $\rm c\grave{a}dl\grave{a}g$  $\mathbb{F}$-adapted Markov process. Let $(P_t)_{t>0}$ be the transition semigroup of generalized solution $\tilde{X}$. (\ref{4.2}) implies that $(P_t)_{t>0}$ is Feller on $H^{-1}(D)$, that is, $P_t\phi\in C_b(H^{-1}(D))$ for $\phi\in C_b(H^{-1}(D))$, $t\geq0$.
\end{rmk}

\indent One of the reasons for considering $(P_t)_{t>0}$ is, if $(Q_t)_{t>0}$ has an invariant measure, then we can construct a invariant measure for $(P_t)_{t>0}$.
\begin{lem}\label{[4.3]} Assume that $\mu$ is an invariant measure of $(Q_t)_{t>0}$, then there exists a measure $\tilde{\mu}$ on $(H^{-1}(D),\mathcal{B}(H^{-1}(D)))$ such that $\tilde{\mu}=\mu$ on $\mathcal{B}(L^2(D))$ and $\tilde{\mu}$ is an invariant measure of $(P_t)_{t\geq0}$.
\end{lem}
\begin{proof} Since $\mathcal{B}(L^2(D))=\mathcal{B}(H^{-1}(D))\cap L^2(D)$, we define $\tilde{\mu}(A):=\mu(A\cap L^2(D))$ for $\forall A\in\mathcal{B}(H^{-1}(D))$. Clearly $\tilde{\mu}$ is well-defined and with full measure on $L^2(D)$, $\mu=\tilde{\mu}$ on $\mathcal{B}(L^2(D))$. Meanwhile, for any $\phi\in B_b(H^{-1}(D))$, $\phi|_{L^2(D)}\in B_b(L^2(D))$. Thus for $\forall t\in[0,T]$ we have
\begin{align}\label{4.3} \int_{H^{-1}(D)}\phi(x)\tilde{\mu}(dx)&=\int_{L^2(D)}\phi(x){\mu}(dx)=\int_{L^2(D)}Q_t\phi(x){\mu}(dx)\notag\\
&=\int_{L^2(D)}P_t\phi(x){\mu}(dx)=\int_{H^{-1}(D)}P_t\phi(x)\tilde{\mu}(dx),
\end{align}
where we use the fact that $Q_t\phi(x)=P_t\phi(x)$ when $x\in L^2(D)$, which is a consequence of Remark \ref{[4.2]} (i). (\ref{4.3}) says that $\tilde{\mu}$ is an invariant measure of $(P_t)_{t>0}$.
\end{proof}
\indent By Lemma \ref{[4.3]}, we see that if the invariant measure of $(P_t)_{t>0}$ is unique, then so does $(Q_t)_{t>0}$. In the rest of this subsection, we will show that the invariant measure of $(P_t)_{t>0}$ is unique. To this end, we will use the so called ``$e$-property" method. Now we introduce the ``$e$-property" and irreducibility.
\begin{defn}\label{[4.4]}{\rm\cite{MR2663632}} We say that the semigroup $(P_t)_{t\geq0}$ has the ``$e$-property" if for $\forall\psi\in {\rm Lip}_b(H^{-1}(D))$, the family of functions $(P_t\psi)_{t\geq0}$ is equicontinuous at every point $x\in H^{-1}(D)$, that is, for any $x\in H^{-1}(D)$ and $\epsilon>0$, there exists $\delta>0$ such that $|P_t\psi(x)-P_t\psi(z)|<\epsilon$ for $\forall t>0$ and any $z\in B(x,\delta):=\{y\in H^{-1}(D):||y-x||_{-1}\leq\delta\}$.
\end{defn}
\begin{defn}\label{[4.5]} We say that $(P_t)_{t\geq0}$ is irreducible if for any $t>0,r>0$ and any $x_0,x_1\in H^{-1}(D)$, we have $P_t\mathbbm{1}_{B(x_1,r)}(x_0)>0$, or equivalently, we have $\mathbb{P}(||\tilde{X}(t,x_0)-x_1||_{-1}> r)<1$.
\end{defn}
\begin{defn}\label{[4.6]} {\rm\cite{MR2904635}} $(P_t)_{t\geq0}$ is said to be weakly topologically irreducible if for any $x_1,x_2\in H^{-1}(D)$, there exists $y\in H^{-1}(D)$ such that for any open set $A$ containing y, there exist $t_1,t_2>0$ with $P_{t_i}\mathbbm{1}_A(x_i)>0$ for $i=1,2$.
\end{defn}

\indent It is obvious that $(P_t)_{t\geq0}$ is weakly topologically irreducible if it is irreducible. Next, we state a criterion on the uniqueness of invariant measures (see Theorem 2 in \cite{MR2904635}).
\begin{lem}\label{[4.7]}{\rm\cite{MR2904635}} If $(P_t)_{t\geq0}$ is weakly topologically irreducible and has the ``$e$-property", then it has at most one invariant measure.
\end{lem}
\begin{prop}\label{[4.8]} Under condition (\textbf{C2}), $(P_t)_{t\geq0}$ satisfies e-property and irreduciblity, hence the invariant measure of $(P_t)_{t\geq0}$ (or $(Q_t)_{t\geq0}$) is unique if it exists.
\end{prop}
\indent Here we remark that condition (\textbf{C1}) is not used in the proof of the uniqueness of invariant measures.

\begin{proof} Since for $\forall\psi\in {\rm Lip}_b(H^{-1}(D))$, we have
\begin{align}\label{4.4} |P_t\psi(x_1)-P_t\psi(x_2)|&=|\mathbb{E}\psi(X(t,x_1))-\mathbb{E}\psi(X(t,x_2))|\notag\\
&\leq K_\phi||x_1-x_2||_{-1},\;\forall t\geq0,
\end{align}
where $K_\phi$ is the Lipschitz coefficient of $\phi$, and the ``$e$-property'' of $(P_t)_{t\geq0}$ follows.

To prove the irreducibility of $(P_t)_{t\geq0}$, we will use Theorem 2.2 established in \cite{2207.11488}. It suffices to verify that Assumption 2.1, 2.2 and 2.4 in \cite{2207.11488} are fulfilled.\\
\indent For Assumption 2.1, we need to verify that the generalized solution $\tilde{X}=(\tilde{X}
(t,\tilde{x}))_{t\geq0}$ forms a strong Markov process. However, in Remark 4.2 we claimed that $\tilde{X}$ is $\rm c\grave{a}dl\grave{a}g$  and moreover, it has the Feller property. Hence $\tilde{X}$ is a strong Markov process.\\
\indent For Assumption 2.2 we need to check that for any $x\in H^{-1}(D)$ and $\eta>0$, there exist $\delta,t>0$ such that
\begin{equation}\label{4.5}\inf_{y\in B(x,\delta)}\mathbb{P}(\tau_{y,x}^\eta\geq t)>0,
\end{equation}
where $\tau_{y,x}^\eta:=\inf{\{t\geq0:\;\tilde{X}(t,y)\notin B(x,\eta)\}}$ and $B(x,\eta)$ is a ball in $H^{-1}(D)$ centered at $x$ with radius $\eta$. Or equivalently, we show that
\begin{equation}\label{4.6} \sup_{y\in B(x,\delta)}\mathbb{P}\left(\sup_{s\in[0,t)}||\tilde{X}(s,y)-x||_{-1}\geq\eta\right)<1.
\end{equation}
Note that
\begin{align}\label{4.7} &\mathbb{P}\left(\sup_{s\in[0,t)}||\tilde{X}(s,y)-x||_{-1}\geq\eta\right)\notag\\
\leq&\mathbb{P}\left(\sup_{s\in[0,t)}||\tilde{X}(s,y)-\tilde{X}(s,x)||_{-1}\geq{\frac{\eta}{2}}\right)
+\mathbb{P}\left(\sup_{s\in[0,t)}||\tilde{X}(s,x)-x||_{-1}\geq{\frac{\eta}{2}}\right)\notag\\
\leq&\frac{4}{\eta^2}\mathbb{E}[\sup_{s\in[0,t)}||\tilde{X}(s,y)-\tilde{X}(s,x)||_{-1}^2]
+\mathbb{P}\left(\sup_{s\in[0,t)}||\tilde{X}(s,x)-x||_{-1}\geq{\frac{\eta}{2}}\right)\notag\\
\leq&\frac{4}{\eta^2}||x-y||_{-1}^2+\mathbb{P}\left(\sup_{s\in[0,t)}||\tilde{X}(s,x)-x||_{-1}\geq{\frac{\eta}{2}}\right),
\end{align}
since $\tilde{X}$ is an $H^{-1}(D)$-valued $\rm c\grave{a}dl\grave{a}g$  process and $y\in B(x,\delta)$, we can choose $t,\delta$ small enough such that the right-hand side of (\ref{4.7}) is strictly less than one, and (\ref{4.6}) follows. Thus Assumption 2.2 is fulfilled.\\
\indent Assumption 2.4 is equivalent to condition (\textbf{C2}) (see Section 4.1 in \cite{2207.11488}).
Therefore, we conclude from Theorem 2.2 in \cite{2207.11488} that $(P_t)_{t\geq0}$ is irreducibility. The proof is complete.
\end{proof}
\subsection{Existence of invariant measures}
\indent Next, we will prove the existence of invariant measures of $(Q_t)_{t\geq0}$, condition (\textbf{C1}) will be used in this subsection. Inspired by \cite{MR2853105}, we introduce the following functional
$$f(u):=(1+||u||_{-1}^2)^{{\frac{\alpha}{2}}},\; \forall u\in H^{-1}(D),$$
where $\alpha\in(0,2)$ is the constant appeared in (\textbf{C1}). The first-order and second-order Frech$\rm\acute{e}$t derivative of $f$ are as follows:
$$Df(u)=\frac{\alpha u}{(1+||u||_{-1}^2)^{1-{\frac{\alpha}{2}}}}\in H^{-1}(D),\forall u\in H^{-1}(D)$$
$$[D^2f(u)]x=\frac{\alpha x}{(1+||u||_{-1}^2)^{1-{\frac{\alpha}{2}}}}-
\frac{(2-\alpha)u\left<u,x\right>_{H^{-1}}}{(1+||u||_{-1}^2)^{2-{\frac{\alpha}{2}}}},\; \forall u,x\in H^{-1}(D),$$
where $D^2f(\cdot)$ is a bounded linear operator from $H^{-1}(D)$ to $H^{-1}(D)$. We list some of the properties of $f$ in the following lemma which will be used later.
\begin{lem}\label{[4.9]} For $\alpha\in(0,1]$ and any $u,v\in H^{-1}(D)$,\\
$\rm(i)$  $|f(u)-f(v)|\leq \left|(1+||u||_{-1}^2)^{{\frac{1}{2}}}-(1+||v||_{-1}^2)^{{\frac{1}{2}}}\right|^\alpha\leq ||u-v||_{-1}^\alpha$.\\
$\rm(ii)$ $|f(u)-f(v)|\leq \alpha ||u-v||_{-1}$.\\
$\rm(iii)$ $||D^2f(u)||_{op}\leq 2$, where $||\cdot||_{op}$ is the operator norm.
\end{lem}
\begin{proof} (i) The first inequality follows from the elementary inequality for $\alpha\in(0,1]$,
$$(a+b)^\alpha\leq a^\alpha+b^\alpha,\; \forall a,b>0,$$
The second inequality follows from the Lipschitz continuity of $\sqrt{1+x^2}$;\\
(ii) For $x,z\in H^{-1}(D)$, by Taylor expansion there exists $\xi\in H^{-1}(D)$ such that
\begin{equation}\label{4.8} |f(x+z)-f(x)|=|\left<Df(\xi),z\right>_{H^{-1}}|
=\frac{\alpha|\left<\xi,z\right>_{H^{-1}}|}{(1+||\xi||_{-1}^2)^{1-{\frac{\alpha}{2}}}}\leq \alpha||z||_{-1},
\end{equation}
where we used the fact that for $\alpha\in(0,1]$,
$$ r\leq (1+r^2)^{\frac{1}{2}}\leq (1+r^2)^{1-{\frac{\alpha}{2}}},\quad\forall r>0.$$
(iii) Clearly we have
\begin{equation}\label{4.9} ||D^2f(u)||_{op}\leq \frac{\alpha}{(1+||u||_{-1}^2)^{1-{\frac{\alpha}{2}}}}
+\frac{(2-\alpha)||u||_{-1}^2}{(1+||u||_{-1}^2)^{2-{\frac{\alpha}{2}}}}\leq \alpha+(2-\alpha)=2,
\end{equation}
thus (iii) follows.
\end{proof}

\indent In the rest of this section, we may also consider an another functional, that is $h(u)=(1+|u|_2^2)^{\alpha/2}$ for $u\in L^2(D)$. We remark that $h$ also has properties which are similar to those listed in Lemma \ref{[4.9]}.\\
\indent There are several ways to prove the existence of invariant measures of $(Q_t)_{t\geq0}$. For example, we may prove that $P_t$ has an invariant measure using the Krylov-Bogoliubov criteria. However, due to the ``big jumps" of $L_t$, it is not easy to prove that the invariant measure obtained in this way is supported on $D(A)$. Thus we will use the Yosida approximation inspired by \cite{MR1942322}, which will not only prove the existence of invariant measures of $(Q_t)_{t\geq0}$, but also implies that the invariant measure is supported on $D(A)$. Recall that in Section 3.3, we defined $X_\epsilon$ as the solution to (\ref{3.7}), and $(Q^\epsilon_t)_{t\geq0}$ as the transition semigroup of $X_\epsilon$. Next we will prove that $(Q^\epsilon_t)_{t\geq0}$ has an invariant measure, and then we obtain the existence of invariant measures for $(Q_t)_{t\geq0}$ by approximation arguments.


\begin{lem}\label{[4.10]} Under condition (\textbf{C1}), the semigroup $(Q^\epsilon_t)_{t\geq0}$ admits a unique invariant measure.
\end{lem}
\begin{proof} Without loss of generality we assume that $\alpha\in(0,1]$ in condition (\textbf{C1}), since if $\alpha\in(1,2)$, condition (\textbf{C1}) also holds for $\alpha-1\in(0,1)$ due to the fact that $\nu(\{|z|_2>1\})<\infty$, where $\nu$ is the intensity measure of $L_t$.
For any fixed $x\in H^{-1}(D)$, applying It$\rm\hat{o}$'s formula we have
\begin{align}\label{4.10} &f(X_\epsilon(t,x))=f(x)-\epsilon\int_0^t \frac{\alpha|X_\epsilon(s,x)|_2^2}{(1+||X_\epsilon(s,x)||_{-1}^2)^{1-{\frac{\alpha}{2}}}}ds\notag\\
&-\int_0^t\frac{\alpha\left<F_\epsilon X_\epsilon(s,x),X_\epsilon(s,x)\right>_{H^{-1}}}{(1+||X_\epsilon(s,x)||_{-1}^2)^{1-{\frac{\alpha}{2}}}}ds\notag\\
&+\int_0^t \int_{|z|_2\leq1} f(X_\epsilon(s-,x)+z)-f(X_\epsilon(s-,x))\,\tilde{N}(dzds)\notag\\
&+\int_0^t \int_{|z|_2\leq1} f(X_\epsilon(s-,x)+z)-f(X_\epsilon(s-,x))-\left<Df(X_\epsilon(s-,x)),z\right>_{H^{-1}}\nu(dz)ds\notag\\
&+\int_0^t \int_{|z|_2>1} f(X_\epsilon(s-,x)+z)-f(X_\epsilon(s-,x))\,N(dzds),\notag\\
&=: f(x)-I_1(t)-I_2(t)+I_3(t)+I_4(t)+I_5(t),\;\forall t\in[0,T].
\end{align}
By Lemma \ref{[3.3]} (iv) and Lemma \ref{[3.1]} (iv) we have
\begin{equation} \label{4.11}I_2(t)\geq -c_2\alpha t.
\end{equation}
By Lemma \ref{[4.9]} (ii) and the It$\rm\hat{o}$ isometry, it is easy to see that $I_3(t)$ is a martingale. As for $I_4(t)$, we first notice that
\begin{align}\label{4.12} &\Big|f(X_\epsilon(s-,x)+z)-f(X_\epsilon(s-,x))-\left<Df(X_\epsilon(s-,x)),z\right>_{H^{-1}}\Big|\notag\\
&\leq||D^2f(\xi)||_{op}||z||_{-1}^2\leq 2||z||_{-1}^2,
\end{align}
where $\xi\in H^{-1}(D)$, and we used Lemma 4.9 (iii) in the last inequality. Therefore,
\begin{equation}\label{4.13} \mathbb{E}|I_4(t)|\leq 2\int_0^t\int_{|z|_2\leq1}||z||_{-1}^2\nu(dz)ds\leq Ct.
\end{equation}
By Lemma \ref{[4.9]} (i) and condition (\textbf{C1}) we have
\begin{equation}\label{4.14} \mathbb{E}|I_5(t)|\leq \int_0^t\int_{|z|_2>1}||z||_{-1}^\alpha\nu(dz)ds\leq Ct.
\end{equation}

Now, taking expectations on both sides of (\ref{4.10}) we get
\begin{equation}\label{4.15}\mathbb{E}\left[(1+||X_\epsilon(t,x)||_{-1}^2)^{{\frac{\alpha}{2}}}\right]
+\mathbb{E}\int_0^t \frac{\epsilon\alpha|X_\epsilon(s,x)|_2^2}{(1+||X_\epsilon(s,x)||_{-1}^2)^{1-{\frac{\alpha}{2}}}}ds
\leq f(x)+Ct.
\end{equation}
Since for $u\in L^2(D)$,
\begin{equation}\label{4.16} |u|_2^\alpha\leq
\frac{|u|_2^\alpha(1+||u||_{-1}^{2})^{1-{\frac{\alpha}{2}}}}{(1+||u||_{-1}^2)^{1-{\frac{\alpha}{2}}}}
\leq \frac{C(1+|u|_{2}^{2})}{(1+||u||_{-1}^2)^{1-{\frac{\alpha}{2}}}},
\end{equation}
thus by (\ref{4.15}) and (\ref{4.16}) we have for $\forall t\in[0,T]$,
\begin{equation}\label{4.17} \epsilon\alpha\mathbb{E}\int_0^t|X_\epsilon(s,x)|_2^\alpha ds\leq C(1+t),
\end{equation}
where the constant $C$ is independent of $t$ and $\epsilon$.
Let $B_R:=\{x\in H^{-1}(D):|x|_{2}\leq R\}$ for $R>0$. Then $B_R$ is a compact subset of $H^{-1}(D)$. For any $T>0$, $\epsilon>0$ and some $x\in H^{-1}(D)$, consider the occupation measure $\pi_{T}^{\epsilon,x}$, for $\forall A\in\mathcal{B}(H^{-1}(D))$
$$\pi_{T}^{\epsilon,x}(A):=\frac{1}{T}\int_0^T Q^\epsilon_t\chi_A(x)dt=\frac{1}{T}\int_0^T \mathbb{P}(X_\epsilon(t,x)\in A)dt.$$
By Chebyshev's inequality and (\ref{4.17}), we have for any $T>1$
\begin{equation}\label{4.18} \pi_{T}^{\epsilon,x}(B_R^c)=\frac{1}{T}\int_0^T \mathbb{P}(X_\epsilon(t,x)\in B_R^c)dt
\leq \frac{1}{TR^\alpha}\int_0^T \mathbb{E}|X_\epsilon(t,x)|^\alpha_{2} dt\leq \frac{C}{R^\alpha},
\end{equation}
where $C$ is independent of $T$. Taking $R$ sufficiently large in (\ref{4.18}) yields that the family $\{\pi_{T}^{\epsilon,x}\}_{T>1}$ is tight. By the Prokhorov Theorem, there exists a sequence $(T_n)_{n\in\mathbb{N}}$ such that as $T_n\to\infty$, $\pi_{T_n}^{\epsilon,x}$ weakly converges to some probability measure $\mu_\epsilon$. According to the Krylov-Bogoliubov Theorem, $\mu_\epsilon$ is an invariant measure of $(Q^\epsilon_t)_{t\geq0}$.\\
\indent Next we show the uniqueness of invariant measures of $(Q^\epsilon_t)_{t\geq0}$. For $\forall x,y\in H^{-1}(D)$, by the chain rule, we have
\begin{align}\label{4.19} &||X_\epsilon(t,x)-X_\epsilon(t,y)||_{-1}^2\notag\\
=&||x-y||^2_{-1}-2\epsilon\int_0^t|X_\epsilon(s,x)-X_\epsilon(s,y)|_2^2ds\notag\\
&-2\int_0^T\left<F_\epsilon X_\epsilon(s,x)-F_\epsilon X_\epsilon(s,y), X_\epsilon(s,x)-X_\epsilon(s,y)\right>_{H^{-1}}ds\notag\\
=&||x-y||^2_{-1}-2\epsilon\int_0^t|X_\epsilon(s,x)-X_\epsilon(s,y)|_2^2ds\notag\\
&-2\int_0^T\left<F_\epsilon X_\epsilon(s,x)-F_\epsilon X_\epsilon(s,y), J_\epsilon X_\epsilon(s,x)-J_\epsilon X_\epsilon(s,y)\right>_{H^{-1}}ds\notag\\
&-2\epsilon\int_0^t||F_\epsilon X_\epsilon(s,x)-F_\epsilon X_\epsilon(s,y)||_{-1}^2ds\notag\\
\leq&||x-y||^2_{-1}-2\epsilon\int_0^t|X_\epsilon(s,x)-X_\epsilon(s,y)|_2^2ds\notag\\
&-2\int_0^T\left<Z_\epsilon X_\epsilon(s,x)-Z_\epsilon X_\epsilon(s,y), J_\epsilon X_\epsilon(s,x)-J_\epsilon X_\epsilon(s,y)\right>_{L^2}ds\notag\\
\leq& ||x-y||^2_{-1}-2\epsilon\int_0^t||X_\epsilon(s,x)-X_\epsilon(s,y)||_{-1}^2ds,
\end{align}
where we used the definition of $F_\epsilon$ and the monotonicity of $\beta$ (see Lemma \ref{[3.1]}). Hence by Gronwall's inequality we have
\begin{equation}\label{4.20} ||X_\epsilon(t,x)-X_\epsilon(t,y)||_{-1}^2\leq e^{-2\epsilon t}||x-y||_{-1}^2.\end{equation}
From (\ref{4.20}) one can prove that the invariant measure of $(Q^\epsilon_t)_{t\geq0}$ is unique. For example, take any functional $g\in {\rm Lip}_b(H^{-1}(D))$ such that
\begin{equation}\label{2024 eq1} |g(x_1)-g(x_2)|\leq \left(K||x_1-x_2||_{-1}\right)\wedge M,
\quad\forall x_1,x_2\in H^{-1}(D),
\end{equation}
holds for some constants $K,M>0$. Then for $\mu_\epsilon$, an invariant measure of $(Q^\epsilon_t)_{t\geq0}$, we have
\begin{align}\label{2024 eq2} &\left|\mathbb{E}[g(X_\epsilon(t,x))]-\int_{H^{-1}(D)}g(y)\mu_\epsilon(dy)\right|\notag\\
=&\left|\mathbb{E}[g(X_\epsilon(t,x))]-\int_{H^{-1}(D)}\mathbb{E}[g(X_\epsilon(t,y))]\mu_\epsilon(dy)\right|\notag\\
\leq&\int_{H^{-1}(D)}\mathbb{E}\left|g(X_\epsilon(t,x)-g(X_\epsilon(t,y)\right|\mu_\epsilon(dy)\notag\\
\leq&\int_{H^{-1}(D)} \Big(\left(K e^{-\epsilon t} ||x-y||_{-1}\right)\wedge M\Big)\mu_\epsilon(dy)\to 0,\quad(t\to\infty)
\end{align}
where we used (\ref{4.20}) and the dominated convergence theorem. By the uniqueness of limit, we conclude from (\ref{2024 eq2}) that the invariant measure of $(Q^\epsilon_t)_{t\geq0}$ is unique.
\end{proof}

\indent Denote by $\mu_\epsilon\in \mathcal{M}_1(H^{-1}(D))$ the invariant measure of $(Q^\epsilon_t)_{t\geq0}$. Next, we will prove two lemmas which will be used to prove the existence of invariant measure of $(Q_t)_{t\geq0}$. Let $\{e_n\}_{n\in\mathbb{N}}$ be the eigenfunctions of $-\Delta$ with Dirichlet boundary conditions in $L^2(D)$, which also constitute an orthonormal basis of $L^2(D)$. For $x\in H^{-1}(D)$ and $m\in\mathbb{N}$, we define the projection operator $P_m$, that is,
$$P_m x:=\sum_{i=1}^m \tensor*[_{H^{-1}}]{\left<x,e_i\right>}{_{H_0^1}}e_i, \quad x\in H^{-1}(D).$$
\begin{lem}\label{v1.2 th1} Under condition (\textbf{C1}), for $\epsilon>0$, $\mu_\epsilon$ is supported on $L^2(D)$. Furthermore, we have $$\int_{L^2(D)}|y|_2^\alpha \,\mu_\epsilon(dy)<\infty.$$
\end{lem}
\begin{proof} In the proof of Lemma \ref{[4.10]} we have defined the measure family $\{\pi_{T}^{\epsilon,x}\}_{T>0}$, where $x\in H^{-1}(D)$ and $\epsilon>0$. From (\ref{4.17}), it is easy to see that
\begin{equation}\label{v1.2 eq1} \int_{H^{-1}(D)}|y|_2^\alpha\,\pi_T^{\epsilon,x}(dy)
=\frac{1}{T}\int_0^T\mathbb{E}|X_\epsilon(t,x)|_2^\alpha dt\leq C_\epsilon,
\end{equation}
where $C_\epsilon>0$ is independent of $T$. Meanwhile, by the monotone convergence theorem we have
\begin{equation}\label{v1.2 eq2} \int_{H^{-1}(D)}|y|_2^\alpha\,\mu_\epsilon(dy)=
\lim_{M\to \infty}\lim_{m\to \infty}  \int_{H^{-1}(D)}(|P_m y|_2^\alpha\wedge M) \,\mu_\epsilon(dy).
\end{equation}
Noticing that the functional $y\to |P_m y|_2^\alpha\wedge M$ is an element in $C_b(H^{-1}(D))$, by the weak convergence of some subsequence $\{\pi_{T_n}^{\epsilon,x}\}_{n\in\mathbb{N}}$ we have
\begin{align}\label{v1.2 eq3} \int_{H^{-1}(D)}(|P_m y|_2^\alpha\wedge M) \,\mu_\epsilon(dy)&=
\lim_{n\to\infty} \int_{H^{-1}(D)}(|P_m y|_2^\alpha\wedge M) \,\pi_{T_n}^{\epsilon,x}(dy)\notag\\
&\leq \lim_{n\to\infty}\int_{H^{-1}(D)}|y|_2^\alpha \,\pi_{T_n}^{\epsilon,x}(dy)\leq C_\epsilon.
\end{align}
Now we conclude from (\ref{v1.2 eq2}) and (\ref{v1.2 eq3}) that Lemma \ref{v1.2 th1} holds.
\end{proof}
\begin{lem}\label{v1.2 th2} For any $\varphi\in C_b(L^2(D))$, there exists a sequence $\{\varphi_m\}_{m\in\mathbb{N}}\subset C_b(H^{-1}(D))$, such that $\sup_{m\in\mathbb{N}}\sup_{x\in H^{-1}(D)} |\varphi_m(x)|<\infty$, and for any $x\in L^2(D)$, we have $\varphi_m(x)\to\varphi(x)$ as $n\to\infty$.
\end{lem}
\begin{proof} Let $\varphi_m(x):=\varphi(P_m x)$ for $x\in H^{-1}(D)$. It is easy to see that the sequence $\{\varphi_m\}_{m\in\mathbb{N}}$ satisfies the required property.
\end{proof}

\indent Now we are going to prove that the semigroup $(Q_t)_{t\geq0}$ has an invariant measure, and the proof is inspired by \cite{MR1942322}. To this end we will show that the family $\{\mu_\epsilon\}_{\epsilon>0}$ is tight, and passing to the limit we will obtain an invariant measure of $(Q_t)_{t\geq0}$.
\begin{prop}\label{[4.13]} Under Conditions (\textbf{C1}),(\textbf{C2}), the semigroup $(Q_t)_{t\geq0}$ has an invariant measure.\end{prop}
\begin{proof}
Let $X_0'$ be an $\mathcal{F}_0$-measurable random variable with distribution $\mu_\epsilon$. By Lemma \ref{v1.2 th1} we know that $X_0'$ is $L^2(D)$-valued. By Lemma \ref{[3.6]}, we see that the map $x\to X_\epsilon(\cdot,x)$ is continuous on $L^2(D)$. Thus $(X_\epsilon(t,X_0'))_{t\in[0,T]}$ is a solution to equation (\ref{2.3}) with initial data $X_0'$. Since $\mu_\epsilon$ is an invariant measure of $(Q_t^\epsilon)_{t\geq0}$, for any $t>0$, the distribution of $X_\epsilon(t,X_0')$ is also $\mu_\epsilon$. \\
\indent Consider the functional $h(u)=(1+|u|_2^2)^{\frac{\alpha}{2}}$ for $u\in L^2(D)$. We apply It\^{o}'s formula to $h(X_\epsilon(t,X_0'))$, and by the similar arguments as we used to prove (\ref{4.15}), we get that
\begin{align}\label{v1.2 eq4} \mathbb{E}[h(X_\epsilon(t,X_0'))]+&\alpha\mathbb{E}\int_0^t
\frac{\epsilon||X_\epsilon(s,X_0')||_1^2+\gamma||Z_\epsilon X_\epsilon(s,X_0')||_1^2+\epsilon|F_\epsilon X_\epsilon(s,X_0')|_2^2}{(1+|X_\epsilon(s,X_0')|_{2}^2)^{1-\frac{\alpha}{2}}}ds\notag\\
&\leq \mathbb{E}[h(X_0')]+Ct.
\end{align}
By Lemma \ref{v1.2 th1}, we have $\mathbb{E}[h(X_0')]<\infty$.
Therefore, we conclude from (\ref{v1.2 eq4}) that

\begin{equation}\label{4.43} \int_{L^2(D)}\frac{\epsilon||x||_1^2+\gamma||Z_\epsilon x||_1^2+\epsilon|F_\epsilon x|_2^2}{(1+|x|_{2}^2)^{1-\frac{\alpha}{2}}}\mu_\epsilon(dx)\leq C,
\end{equation}
where $C$ is independent of $\epsilon>0$ and $t>0$.
Now, since $x=\epsilon F_\epsilon x+J_\epsilon x$, we have
\begin{align}\label{4.44} |x|_2^2&=\left<J_\epsilon x,x\right>_{L^2}+\left<\epsilon F_\epsilon x,x\right>_{L^2}\notag\\
&\leq \frac{1}{2}\left(|J_\epsilon x|_2^2+|x|_2^2+\epsilon |F_\epsilon x|_2^2+\epsilon|x|_2^2\right).
\end{align}
Note that $J_\epsilon x=(\beta+\epsilon I)^{-1}Z_\epsilon x$, by Remark \ref{[3.2]} (ii) there exist $C>0$, independent of $\epsilon$, such that
\begin{equation}\label{4.45} |J_\epsilon x|_2\leq C(1+|Z_\epsilon x|_2).
\end{equation}
Combining (\ref{4.44}) and (\ref{4.45}) together yields that for $\epsilon\in(0,\epsilon_0)$,
\begin{align}\label{4.46} |x|_2^2&\leq C(1+|Z_\epsilon(x)|_2^2+\epsilon |F_\epsilon x|_2^2)\leq C(1+||Z_\epsilon(x)||_1^2+\epsilon |F_\epsilon x|_2^2),
\end{align}
where we used the Poincar$\rm\acute{e}$ inequality in last step. From (\ref{4.16}), (\ref{4.43}) and (\ref{4.46}) it follows that
\begin{equation}\label{4.47}  \int_{L^2(D)} |x|_2^\alpha \mu_\epsilon(dx)\leq C\int_{L^2(D)} \frac{1+|x|_2^2}{(1+|x|_{2}^2)^{1-\frac{\alpha}{2}}}\mu_\epsilon(dx)\leq C,
\end{equation}
where $C$ is independent of $\epsilon$.
Therefore, the family $\{\mu_\epsilon\}_{\epsilon>0}$ is tight in $\mathcal{M}_1(H^{-1}(D))$ due to the compact embedding $L^2(D)\subset H^{-1}(D)$. By the Prokhorov Theorem, there exists a subsequence $\{\mu_{\epsilon_n}\}_{n\in\mathbb{N}}$ such that $\mu_{\epsilon_n}$ converges weakly to some probability measure $\mu$ on $H^{-1}(D)$. Moreover, similar to Lemma \ref{v1.2 th1}, we can prove that
\begin{equation}\label{4.48} \int_{L^2(D)} |x|_2^\alpha \mu(dx)<\infty,
\end{equation}
which implies that $\mu$ is supported on $L^2(D)$.\\
\indent Recalling that $\mu_\epsilon$ is the invariant measure of $(Q_t^\epsilon)_{t\geq0}$, and $\mu_\epsilon$ is supported on $L^2(D)$, thus we have for any $\epsilon>0$,
\begin{equation}\label{4.41} \int_{L^2(D)} Q^\epsilon_t\phi(x)\mu_\epsilon(dx)=\int_{L^2(D)}\phi(x)\mu_\epsilon(dx),
 \,\forall \phi\in C_b(L^2(D)).
\end{equation}
Next, we will take $\epsilon_n\to0$ on both side of (\ref{4.41}) to show that $\mu$ is a invariant measure of $(Q_t)_{t\geq0}$. Since $\mu_{\epsilon_n}$ weakly converges to $\mu$, we have
\begin{equation}\label{4.49} \lim_{n\to\infty}\int_{L^2(D)}\phi(x)\mu_{\epsilon_n}(dx)=\int_{L^2(D)}\phi(x)\mu(dx),\;
\forall \phi\in {\rm Lip}_b(H^{-1}(D)).
\end{equation}
As for the left-hand side of (\ref{4.41}), we have
\begin{align}\label{4.50} \int_{L^2(D)} (Q^{\epsilon_n}_t\phi(x)\mu_{\epsilon_n}(dx)&=\int_{L^2(D)} \left(Q^{\epsilon_n}_t\phi(x)-Q_t\phi(x)\right)\mu_{\epsilon_n}(dx)\notag\\
&+\int_{L^2(D)}Q_t\phi(x)\mu_{\epsilon_n}(dx).
\end{align}
Let $V=\{x\in H^{-1}(D): |x|_2\leq R\}$ for some $R>0$. By (\ref{4.47}) we see that for any $\eta>0$, there exists a $R>0$ such that $\mu_{\epsilon_n}(V^c)<\eta$ for any $n\in\mathbb{N}$. Thus by Lemma \ref{[4.12]} and the arbitrariness of $\eta$ we obtain that
\begin{equation}\label{4.51} \lim_{n\to\infty}\int_{L^2(D)}
\left(Q^{\epsilon_n}_t\phi(x)-Q_t\phi(x)\right)\mu_{\epsilon_n}(dx)=0, \quad \forall \phi\in {\rm Lip}_b(H^{-1}(D)).
\end{equation}
On the other hand, since $Q_t\phi(x)=P_t\phi(x)$ for $x\in L^2(D)$, where $(P_t)_{t>0}$ is the transition semigroup of the generalized solution to (\ref{2.3}), see Remark \ref{[4.2]} (iii). 
And by the Feller property of $P_t$ we have $P_t\phi\in C_b(H^{-1}(D))$, thus by the weak convergence of $\mu_{\epsilon_n}$ we have
\begin{equation}\label{4.52} \lim_{n\to\infty}\int_{L^2(D)}Q_t\phi(x)\mu_{\epsilon_n}(dx)
=\int_{L^2(D)}Q_t\phi(x)\mu(dx).
\end{equation}
Now, taking into accout (\ref{4.49})$\sim$(\ref{4.52}) we conclude that for any $\phi\in {\rm Lip}_b(H^{-1}(D))$,
\begin{equation}\label{4.53} \int_{L^2(D)}Q_t\phi(x)\mu(dx)=\int_{L^2(D)}\phi(x)\mu(dx).
\end{equation}
Since ${\rm Lip}_b(H^{-1}(D))$ is dense in $C_b(H^{-1}(D))$, we see that (\ref{4.53}) holds for any $\phi\in C_b(H^{-1}(D))$. By Lemma \ref{v1.2 th2} and the dominated convergence theorem, (\ref{4.53}) holds also for any $\phi\in C_b(L^2(D))$. Therefore, $\mu$ is an invariant measure of $(Q_t)_{t\geq0}$.
\end{proof}

\indent Now we have proved that $(Q_t)_{t\geq0}$ has a unique invariant measure $\mu$. Since $Q_t\phi(x)=P_t\phi(x)$ for $\phi\in C_b(H^{-1}(D))$ and $x\in L^2(D)$, by (\ref{4.53}) and Proposition \ref{[4.8]} we see that $\mu$ is also the unique invariant measure of $(P_t)_{t\geq0}$. Recall that the generalized solution $\tilde{X}$ is the variational solution to (\ref{2.3}) with Gelfand triple
$L^2(D)\subset H^{-1}(D)\subset (L^2(D))^*$, and the ``coercivity condition" (see \cite{MR3158475}) makes sure that $(P_t)_{t\geq0}$ satisfies the convergence property mentioned in Theorem \ref{[2.4]}.
\begin{prop}\label{[new 4.12]} Under Condition (\textbf{C1}),(\textbf{C2}), for any $\nu_0\in\mathcal{M}_1(L^2(D))$ with $\int_{L^2} |x|_2^\alpha \nu_0(dx)<\infty$, we have as $T\to\infty$,
$$\frac{1}{T}\int_0^T Q^*_s\nu_0 ds\Rightarrow \mu$$
in weak topology of $\mathcal{M}_1(H^{-1}(D))$.
\end{prop}
\begin{proof} Let $\tilde{X}_0$ be an $\mathcal{F}_0$-measurable random variable with distribution $\nu_0$. From (\ref{4.2}) we see that the map $x\to \tilde{X}(\cdot,x)$ is continuous on $H^{-1}(D)$. Thus $(\tilde{X}(t,\tilde{X}_0))_{t\geq0}$ is a generalized solution to (\ref{2.3}) with initial data $\tilde{X}_0$. Analogous to Lemma \ref{[4.10]}, applying It\^{o}'s formula to $(1+||\tilde{X}(t,\tilde{X}_0)||_{-1}^2)^{\frac{\alpha}{2}}$ we can show that
\begin{equation}\label{new 4.45} \frac{1}{T}\int_0^T\mathbb{E}|\tilde{X}(t,\tilde{X}_0)|_2^\alpha dt \leq C,
\end{equation}
where $C$ is independent of $T>0$. Define
$$R_T^*{\nu_0}=\frac{1}{T}\int_0^T P^*_t\nu_0 ds.$$
By (\ref{new 4.45}) and Chebyshev's inequality we see that $(R_T^*{\nu_0})_{T\geq0}$ is tight in $\mathcal{M}_1(H^{-1}(D))$. Since $(P_t)_{t\geq0}$ has a unique invariant measure $\mu$, by the Krylov-Bogoliubov Theorem we can deduce that
$R_T^*{\nu_0}\Rightarrow \mu$ when $T\to\infty$. Again, by the fact that $Q_t\phi(x)=P_t\phi(x)$ for $\phi\in C_b(H^{-1}(D))$ and $x\in L^2(D)$, we have $P^*_s\nu_0=Q^*_s\nu_0$ for any $s>0$, hence Proposition \ref{[new 4.12]} follows.
\end{proof}

\subsection{Support of invariant measure $\mu$}
\indent Finally, we will show that $\mu$ is supported on $D(A):=\{x\in L^2(D):\, \beta(x)\in H_0^1(D)\}$. Recalling that when $x\in D(A)$, $Ax=-\Delta\beta(x)$, when $x\notin D(A)$, we define $||Ax||_{-1}:=\infty$. Since the operator $A$ is m-accretive (see Section 3.1 in \cite{MR2582280}), we can define its Yosida approximation, that is, for $\epsilon>0$ and $x\in H^{-1}(D)$,
$$L_\epsilon x:=(1+\epsilon A)^{-1}x,\; A_\epsilon x:=\frac{1}{\epsilon}(x-L_\epsilon x)=A L_\epsilon x.$$
By the property of the Yosida approximation, we have $||A_\epsilon x||_{-1}\leq ||Ax||_{-1}$ for any $x\in H^{-1}(D)$. Let $J_\epsilon$ and $F_\epsilon$ be the operators defined in Section 3. We start with the following lemmas.
\begin{lem}\label{[4.14]} For $\epsilon>0$ and $ x\in H^{-1}(D)$, we have\\
$\rm(i)$ $\epsilon||J_\epsilon x||_1\leq \frac{1}{2}|x|_2$;\\
$\rm(ii)$ $||F_\epsilon x||_{-1}\leq ||Ax||_{-1}+|x|_2$.
\end{lem}
\begin{proof} If $x\notin L^2(D)$, then $|x|_2=\infty$, and (i),(ii) obviously hold. Hence below we only consider the case that $x\in L^2(D)$.

Proof of (i). By the definition of $J_\epsilon$ we have
\begin{equation}\label{4.54} J_\epsilon x-\epsilon\Delta(\beta+\epsilon I)J_\epsilon x=x.
\end{equation}
Then taking the inner product in $L^2(D)$ with $J_\epsilon x$ on both sides of (\ref{4.54}) we have
\begin{equation}\label{4.55} |J_\epsilon x|_2^2+\epsilon\left<\beta(J_\epsilon x),J_\epsilon x\right>_{H_0^1}
+\epsilon^2||J_\epsilon x||_1^2=\left<x,J_\epsilon x\right>_{L^2}\leq \frac{1}{4}|x|_2^2+|J_\epsilon x|_2^2.
\end{equation}
Since $\beta$ is monotone, with $\beta'(r)\geq0$ we have
$$\left<\beta(J_\epsilon x),J_\epsilon x\right>_{H_0^1}=\left<\beta'(J_\epsilon x)\nabla J_\epsilon x,\nabla J_\epsilon x\right>_{L^2}\geq0,$$
hence (i) follows from (\ref{4.55}).

Proof of (ii). By the definition of $L_\epsilon$ we have
\begin{equation}\label{4.56} L_\epsilon x-\epsilon \Delta\beta(L_\epsilon x)=x.
\end{equation}
From (\ref{4.54}) and (\ref{4.56}) it follows that
\begin{equation}\label{4.57} (J_\epsilon x-L_\epsilon x)-\epsilon\Delta\left(\beta(J_\epsilon x)-\beta(L_\epsilon x)\right)-\epsilon^2\Delta J_\epsilon x=0.
\end{equation}
Taking inner product in $L^2(D)$ with $\beta(J_\epsilon x)-\beta(L_\epsilon x)$ on both side of (\ref{4.57}) yields
\begin{align}\label{4.58} \left<J_\epsilon x-L_\epsilon x, \beta(J_\epsilon x)-\beta(L_\epsilon x)\right>_{L^2(D)}
&+\epsilon||\beta(J_\epsilon x)-\beta(L_\epsilon x)||_1^2\notag\\
&=-\epsilon^2\left<J_\epsilon x,\beta(J_\epsilon x)-\beta(L_\epsilon x)\right>_{H_0^1}.
\end{align}
By Lemma \ref{[3.1]} (iii), we have $\left<J_\epsilon x-L_\epsilon x, \beta(J_\epsilon x)-\beta(L_\epsilon x)\right>_{L^2(D)}\geq0$, thus we have
\begin{equation}\label{4.59} \epsilon||\beta(J_\epsilon x)-\beta(L_\epsilon x)||_1^2\leq
\frac{\epsilon^3}{2}||J_\epsilon x||_1^2+\frac{\epsilon}{2}||\beta(J_\epsilon x)-\beta(L_\epsilon x)||_1^2,
\end{equation}
which implies
\begin{equation}\label{4.60} ||-\Delta\beta(J_\epsilon x)-(-\Delta\beta(L_\epsilon x))||_{-1}
\leq \epsilon||J_\epsilon x||_1.
\end{equation}
Noticing that $F_\epsilon x=-\Delta\beta(J_\epsilon x)-\epsilon\Delta J_\epsilon x$, we have
\begin{align}\label{4.61} ||F_\epsilon x||_{-1}&\leq ||-\Delta\beta(J_\epsilon x)-A(L_\epsilon x)||_{-1}
+||A(L_\epsilon x)||_{-1}+\epsilon||\Delta J_\epsilon x||_{-1}\notag\\
&\leq 2\epsilon||J_\epsilon x||_1+||Ax||_{-1}\leq||Ax||_{-1}+|x|_2,
\end{align}
where we used Lemma \ref{[4.14]} (i), (\ref{4.60}) and the fact that $||A(L_\epsilon x)||_{-1}\leq ||Ax||_{-1}$. Thus (ii) follows from (\ref{4.61}).
\end{proof}
\begin{lem} \label{[4.15]}
$\rm(i)$ For $x\in H^{-1}(D)$, we have $\lim_{\epsilon\to0}J_\epsilon x=x$ in $H^{-1}(D)$.\\
$\rm(ii)$ For $x\in L^2(D)$, we have
$$||Ax||_{-1}\leq \liminf_{\epsilon\to0}||F_{\epsilon} x||_{-1}.$$
\end{lem}
\begin{proof} Proof of (i). By the definition of $J_\epsilon$ and Lemma \ref{[3.3]} we have
\begin{equation}\label{4.62} ||J_\epsilon x-x||_{-1}=||J_\epsilon x-J_\epsilon(x+\epsilon G_\epsilon x)||_{-1}\leq \epsilon||G_\epsilon x||_{-1},
\end{equation}
where $G_\epsilon x=-\Delta(\beta+\epsilon I)x$. Since $\beta+\epsilon I: \mathbb{R}\to\mathbb{R}$ is Lipschitz continuous, there exist $C>0$ and $\epsilon_0>0$, such that for any $\epsilon\in (0,\epsilon_0)$ and $x\in H_0^1(D)$,
\begin{equation}\label{4.63} ||G_\epsilon x||_{-1}=||(\beta+\epsilon I)x||_1\leq C||x||_1.\end{equation}
Thus by (\ref{4.62}) and (\ref{4.63}) we see that $\lim_{\epsilon\to0}J_\epsilon x=x$ if $x\in H_0^1(D)$.
For the general case $x\in H^{-1}(D)$, there exists a sequence $\{x_n\}_{n\in\mathbb{N}}\subset H_0^1(D)$ such that $\lim_{n\to\infty}||x_n-x||_{-1}=0$. Then, by Lemma \ref{[3.3]} (i) we have
\begin{align}\label{4.64} ||J_\epsilon x-x||_{-1}&=||J_\epsilon x-J_\epsilon x_n||_{-1}+||J_\epsilon x_n-x_n||_{-1}
+||x_n-x||_{-1}\notag\\
&\leq 2||x_n-x||_{-1}+||J_\epsilon x_n-x_n||_{-1}. 
\end{align}
First letting $\epsilon\to0$ and then letting $n\to\infty$ yields (i). 

Proof of (ii). Suppose first $x\in D(A)$. By Lemma \ref{[4.14]}, we have for any $\epsilon>0$,
$$||F_\epsilon x||_{-1}\leq ||Ax||_{-1}+|x|_2<\infty.$$
Now, take any sequence $\{\epsilon_n\}_{n\in\mathbb{N}}$ with $\epsilon_n\to0$.
Then there exists a subsequence $\{\epsilon_{n_k}\}_{k\in\mathbb{N}}$ such that $F_{\epsilon_{n_k}}x\stackrel{w}{\rightarrow}y$ for some $y\in H^{-1}(D)$ as $k\to\infty$. By Lemma \ref{[4.15]} (i), $J_{\epsilon_{n_k}}x\to x$ in $H^{-1}(D)$. Thus, $\epsilon_{n_k} J_{\epsilon_{n_k}}x\to 0$ in $H^{-1}(D)$. From Lemma \ref{[4.14]} we conclude that $\epsilon_{n_k} J_{\epsilon_{n_k}}x\stackrel{w}{\rightarrow}0$ in $H_0^1(D)$ (taking a further subsequence but not relabeled). Hence $\epsilon_{n_k} \Delta J_{\epsilon_{n_k}}x\stackrel{w}{\rightarrow}0$ in $H^{-1}(D)$. Since $F_\epsilon x=-\Delta\beta(J_\epsilon x)-\epsilon\Delta J_\epsilon x$, we see that $AJ_{\epsilon_{n_k}} x\stackrel{w}{\rightarrow}y$. Since $A$ is m-accretive, according to Proposition 3.4 in \cite{MR2582280}, $A$ is demiclosed. Therefore, combining with the fact $J_{\epsilon_n}x\to x$ we deduce that $y=Ax$. Hence
$$||Ax||_{-1}=||y||_{-1}\leq \liminf_{k\to\infty}||F_{\epsilon_{n_k}} x||_{-1},$$
and by the arbitrariness of $\{\epsilon_n\}_{n\in\mathbb{N}}$, we derive (ii) in the case that $x\in D(A)$.

Now, consider the case $x\notin D(A)$. Note that
\begin{align} \label{4.65} ||F_\epsilon x||_{-1}&=||A(J_\epsilon x)-A(L_\epsilon x)+A(L_\epsilon x)+\epsilon\Delta J_\epsilon x||_{-1}\notag\\
&\geq\Big|||A(L_\epsilon x)||_{-1}-||A(J_\epsilon x)-A(L_\epsilon x)+\epsilon\Delta J_\epsilon x||_{-1}\Big|.
\end{align}
However, by (\ref{4.60}) and Lemma \ref{[4.14]} (i) we have
\begin{align}\label{4.66} &||A(J_\epsilon x)-A(L_\epsilon x)+\epsilon\Delta J_\epsilon x||_{-1}\notag\\
\leq &||A(J_\epsilon x)-A(L_\epsilon x)||_{-1}+\epsilon||\Delta J_\epsilon x||_{-1}\leq|x|_2<\infty,
\end{align}
and by the property of Yosida approximation we have $\lim_{\epsilon\to0}||A(L_\epsilon x)||_{-1}=\lim_{\epsilon\to0}||A_\epsilon x||_{-1}=\infty$ when $x\notin D(A)$. Thus we deduce from (\ref{4.64}) and (\ref{4.65}) that $\lim_{\epsilon\to0}||F_\epsilon x||_{-1}=\infty$. The proof of (ii) is complete.
\end{proof}

With Lemma \ref{[4.14]} and \ref{[4.15]} in hand, we are ready to prove that the invariant measure of $(Q_t)_{t\geq0}$ is supported on $D(A)$.
\begin{prop}\label{[4.16]} Under conditions (\textbf{C1})(\textbf{C2}), the invariant measure of the semigroup $(Q_t)_{t\geq0}$ is supported on $D(A)$.
\end{prop}
\begin{proof} From the proof of Proposition \ref{[4.13]}, there exists a sequence $\{\epsilon_n\}_{n\in\mathbb{N}}$ such that $\epsilon_n\to0$ and $\mu_{\epsilon_n}$ weakly converges to $\mu$. By Lemma \ref{[4.15]}, we have
$$||Ax||_{-1}\leq \liminf_{n\to\infty} ||F_{\epsilon_n}x||_{-1},\,\forall x\in L^2(D).$$
Note that $F_\epsilon x=-\Delta Z_\epsilon x$. From (\ref{4.43}) and (\ref{4.47}), we have
\begin{equation}\label{4.67}
\int_{L^2(D)}\frac{||F_{\epsilon_n} x||_{-1}^2}{(1+|x|_2^2)^{1-\frac{\alpha}{2}}}\mu_{\epsilon_n}(dx)\leq C,\quad
\int_{L^2(D)}|x|_2^\alpha \mu_{\epsilon_n}(dx)\leq C,
\end{equation}
where $C$ is a constant independent of $n$.
Let $\theta=\alpha(2-\alpha)/4$, by H$\rm\ddot{o}$lder's inequality we have
\begin{align}\label{4.68}&\int_{L^2(D)}||F_{\epsilon_n}x||_{-1}^{\alpha}\mu_{\epsilon_n}(dx)=
\int_{L^2(D)}\frac{||F_{\epsilon_n}x||_{-1}^{\alpha}}{(1+|x|_2^2)^\theta}(1+|x|_2^2)^\theta\mu_{\epsilon_n}(dx)\notag\\
\leq& \left(\int_{L^2(D)}\frac{||F_{\epsilon_n} x||_{-1}^2}{(1+|x|_2^2)^{1-\frac{\alpha}{2}}}\mu_{\epsilon_n}(dx)\right)^{\frac{\alpha}{2}}
\left(\int_{L^2(D)}(1+|x|_2^2)^{\frac{\alpha}{2}} \mu_{\epsilon_n}(dx)\right)^{\frac{2-\alpha}{2}}.
\end{align}
It follows from (\ref{4.67}) and (\ref{4.68}) that
\begin{equation}\label{4.69} \int_{L^2(D)}||F_{\epsilon_n}x||_{-1}^{\alpha}\mu_{\epsilon_n}(dx)\leq C
\end{equation}
for some constant $C$ which is independent of $n$. By Chebyshev's inequality and (\ref{4.67}),(\ref{4.69}), we have there exists a constant $C>0$ independent of $n$ such that for any $M>0$,
\begin{equation}\label{4.70} \mu_{\epsilon_n}(||F_{\epsilon_n}x||_{-1}>M)\leq CM^{-\alpha},\quad
\mu_{\epsilon_n}(|x|_2>M)\leq CM^{-\alpha}.
\end{equation}
Now, let us consider the set $\{x\in H^{-1}(D):||Ax||_{-1}>M\}\subset H^{-1}(D)$ for some $M>0$, by Lemma \ref{[4.15]} and the fact that $\mu$ is supported on $L^2(D)$, we have
\begin{align}\label{4.71} \mu(||Ax||_{-1}>M)&\leq\mu(\liminf_{n\to\infty}||F_{\epsilon_n}x||_{-1}>M)\notag\\
&\leq\mu\left(\bigcup_{N=1}^\infty \bigcap_{n=N}^\infty \left\{||F_{\epsilon_n}x||_{-1}>M\right\}\right)\notag\\
&=\lim_{N\to\infty}\mu\left(\bigcap_{n=N}^\infty \left\{||F_{\epsilon_n}x||_{-1}>M\right\}\right).
\end{align}
For fixed $N\in\mathbb{N}$, we define $a_n^{(N)}=2(n-N)/n$ for $n\geq N$, then $a_n^{(N)}\in[0,2]$ and
$\lim_{n\to\infty}a_n^{(N)}=2$. Note that
\begin{equation}\label{4.72} \bigcap_{n=N}^\infty \left\{||F_{\epsilon_n}x||_{-1}>M\right\}\subset
\bigcap_{n=N}^\infty \left\{||F_{\epsilon_n}x||_{-1}+a_n^{(N)}|x|_2>M\right\}:=A_N.
\end{equation}
\textbf{Claim}: $A_N\subset H^{-1}(D)$ is an open set. \\
Let us first complete the proof of Proposition \ref{[4.16]} accepting the \textbf{Claim}.
Since $A_N$ is an open set, by the weak convergence we have
\begin{equation}\label{4.73} \mu(A_N)\leq\liminf_{m\to\infty}\mu_{\epsilon_m}(A_N).
\end{equation}
But, for $m>N$,
\begin{align}\label{4.74} \mu_{\epsilon_m}(A_N)&\leq
\mu_{\epsilon_m}(||F_{\epsilon_m}x||_{-1}+a_m^{(N)}|x|_2>M)\notag\\
&\leq \mu_{\epsilon_m}(||F_{\epsilon_m}x||_{-1}>\frac{M}{2})+
\mu_{\epsilon_m}(|x|_2>\frac{M}{4})\leq CM^{-\alpha},
\end{align}
where we have used (\ref{4.70}). Therefore, we have $\mu(A_N)\leq CM^{-\alpha}$ for any $N$. Thus from (\ref{4.71}) and (\ref{4.72}) it follows that
\begin{equation}\label{4.75} \mu(||Ax||_{-1}>M)\leq \lim_{N\to\infty}\mu(A_N)\leq CM^{-\alpha}.
\end{equation}
Taking $M\to\infty$ in (\ref{4.75}), we conclude that $\mu$ is supported on $D(A)$.\\
\indent It remains to prove the \textbf{Claim} above. It is equivalent to prove that
$$A_N^c:=\bigcup_{n=N}^\infty\left\{||F_{\epsilon_n}x||_{-1}+a_n^{(N)}|x|_2\leq M\right\}$$
is a close subset of $H^{-1}(D)$. First, it is easy to see that for fixed $n\geq N$, the set
$$\left\{||F_{\epsilon_n}x||_{-1}+a_n^{(N)}|x|_2\leq M\right\}\subset H^{-1}(D)$$
is a close subset, because by Lemma \ref{[3.3]} $F_{\epsilon_n}$ is Lipschitz continuous on $H^{-1}(D)$ and the $L^2$-norm $|\cdot|_2$ is lower semicontinuous on $H^{-1}(D)$.
Now, take any sequence $\{x_k\}_{k\in\mathbb{N}}\subset A_N^c$ such that $\lim_{k\to\infty}||x_k-x||_{-1}=0$ for some $x\in H^{-1}(D)$. For every $k$, there exists $n_k\geq N$ such that
\begin{equation} \label{4.76} ||F_{\epsilon_{n_k}}x_k||_{-1}+a_{n_k}^{(N)}|x_k|_2\leq M.
\end{equation}
We may assume that $\sup_{k}n_k=\infty$, since otherwise we must have
$$\{x_k\}_{k\in\mathbb{N}}\subset\bigcup_{n=N}^{\sup_{k}n_k}
\left\{||F_{\epsilon_n}x||_{-1}+a_n^{(N)}|x|_2\leq M\right\},$$
that is the sequence $\{x_k\}_{k\in\mathbb{N}}$ belongs to the finite union of some close subsets in $H^{-1}(D)$, which implies that $x\in A_N^c$. Taking a subsequence if necessary, we may assume that $n_k\uparrow\infty$. Since $a_{n_k}^{(N)}\to2$ as $k\to\infty$, by (\ref{4.76}) we have $|x_k|_2\leq M$ for $k$ large enough, thus the sequence $\{x_k\}_{k\in\mathbb{N}}$ has a weak convergent subsequence (still denoted by  $x_k$), such that $x_k\stackrel{w}{\rightarrow}x$ in $L^2(D)$. Meanwhile, by (\ref{4.76}) we know that $F_{\epsilon_{n_k}}x_k$ has a  weak convergent subsequence (still denoted by $F_{\epsilon_{n_k}}x_k$) such that $F_{\epsilon_{n_k}}x_k\stackrel{w}{\rightarrow}y$ in $H^{-1}(D)$ for some $y\in H^{-1}(D)$.
By the Lipschitz property of $J_\epsilon$ in $H^{-1}(D)$ (see Lemma \ref{[3.3]}) we have
$$\lim_{k\to\infty}||J_{\epsilon_{n_k}}x_k-J_{\epsilon_{n_k}}x||_{-1}\leq\lim_{k\to\infty}||x_k-x||_{-1}=0.$$
According to Lemma \ref{[4.15]} (i), we have $\lim_{k\to\infty}J_{\epsilon_{n_k}}x=x$ in $H^{-1}(D)$. Since $J_{\epsilon_{n_k}}x_k=(J_{\epsilon_{n_k}}x_k-J_{\epsilon_{n_k}}x)+J_{\epsilon_{n_k}}x$, we see that  $J_{\epsilon_{n_k}}x_k\to x$ in $H^{-1}(D)$. By Lemma \ref{[4.14]} (i), we have
$$\epsilon_{n_k}||-\Delta J_{\epsilon_{n_k}} x_k||_{-1}\leq \frac{1}{2}|x_k|_2\leq \frac{M}{2}.$$
Hence there exists a subsequence (still denoted by $n_k$) such that $-\epsilon_{n_k}\Delta J_{\epsilon_{n_k}}x_k\stackrel{w}{\rightarrow}0$ in $H^{-1}(D)$. Note that
$$A(J_{\epsilon_{n_k}}x_k)=F_{\epsilon_{n_k}}x_k+\epsilon_{n_k}\Delta J_{\epsilon_{n_k}}x_k,$$
We conclude that we have $A(J_{\epsilon_{n_k}}x_k)\stackrel{w}{\rightarrow}y$ in $H^{-1}(D)$.
Since $A$ is demiclosed, it follow that $y=Ax$. Therefore, we have $F_{\epsilon_{n_k}}x_k\stackrel{w}{\rightarrow}Ax$ in $H^{-1}(D)$.
Now, by Lemma \ref{[4.14]} (ii),
\begin{align}\label{4.77} ||F_{\epsilon_N}x||_{-1}&\leq||Ax||_{-1}+|x|_2\notag\\
&\leq \liminf_{k\to\infty}\left(||F_{\epsilon_{n_k}}x_k||_{-1}+|x_k|_2\right)\notag\\
&\leq M,
\end{align}
where in second step we used the property of weak convergence, and in last step we used (\ref{4.76}) and the fact that $a_{n_k}^{(N)}>1$ when $k$ is large enough. Since $a_{N}^{(N)}=0$, by (\ref{4.77}), we see that $x\in A_N^c$, which proves that $A_N$ is an open set. The proof of the \textbf{Claim} is complete.
\end{proof}

{Now, putting Proposition \ref{[4.8]}, \ref{[4.13]}, \ref{[new 4.12]} and \ref{[4.16]} together, one completes the proof of Theorem \ref{[2.4]}.}

\section{Acknowledgement}
This work is partially supported by National Key R\&D Program of China (No. 2022
YFA1006001), National Natural Science Foundation of China (No. 12131019, 12371151,
11721101), and  the Fundamental Research Funds for the Central Universities(No. WK3470000031, WK0010000081).

\bibliographystyle{plain}

\end{document}